\documentclass[final,3p,times,nonatbib]{elsarticle}
\usepackage{graphicx}
\usepackage{amsmath}
\usepackage{amsthm}
\usepackage{amsfonts}
\usepackage{physics}
\usepackage{multirow}
\usepackage{bm}
\usepackage{empheq}
\usepackage{mathtools}
\mathtoolsset{showonlyrefs,showmanualtags}

\newtheorem{thm}{Theorem}
\newtheorem{dfn}[thm]{Definition}
\newtheorem{prop}[thm]{Proposition}
\newtheorem{lem}[thm]{Lemma}

\newcommand{\dsv}[2]{{{#1}^{\rm #2}}}
\newcommand{\dsh}[2]{{{#1}_{\rm #2}}}
\newcommand{\ctext}[1]{\raise0.2ex\hbox{\textcircled{\scriptsize{#1}}}}
\newcommand{\gbi}{{G_{\rm bi}}}
\newcommand{\gful}{{G_{\rm full}}}
\newcommand{\gred}{{G_{\rm red}}}

\title{Structural Inconsistency and Stability Classification of Multi-symplectic Diamond Schemes}
\author[1]{Kaito Sato\corref{cor1}}
\ead{ksato-27@g.ecc.u-tokyo.ac.jp}
\author[2]{Shun Sato}
\ead{shun@mist.i.u-tokyo.ac.jp}
\author[1]{Takayasu Matsuo}
\ead{matsuo@mist.i.u-tokyo.ac.jp}
\cortext[cor1]{Corresponding author}
\affiliation[1]{organization={Department of Mathematical Informatics, Graduate School of Information Science and Technology, The University of Tokyo},
postcode={113-8656},
city={Tokyo},
country={Japan}}
\affiliation[2]{organization={Department of Mathematical Sciences, Graduate School of Science, Tokyo Metropolitan University},
postcode={192-0397},
city={Tokyo},
country={Japan}}

\date{August 2025}

\begin{document}

\begin{abstract}
Multi-symplectic diamond schemes proposed by McLachlan and Wilkins (2015) provide a framework for the numerical integration of Hamiltonian partial differential equations, combining local implicitness with high-order accuracy and discrete multi-symplectic conservation laws. Despite these advantages, their behavior beyond a limited class of model equations remains poorly understood, and numerical difficulties may arise depending on the underlying multi-symplectic formulation.
In this paper, we present a systematic stability analysis framework for diamond schemes applied to general multi-symplectic PDEs. The approach consists of three stages. First, we identify structural inconsistency of the local diamond update using Dulmage--Mendelsohn decomposition, revealing cases in which the scheme is intrinsically unsolvable. Second, we introduce a graph-based error-propagation analysis that yields a necessary stability condition by detecting negative cycles in a weighted directed graph. Third, for equations that pass the preliminary tests, we derive eigenvalue-based timestep restrictions providing sufficient conditions for stability.
The analysis leads to a comprehensive classification of multi-symplectic PDEs according to whether diamond schemes are structurally inconsistent, unconditionally unstable, or conditionally stable. In particular, we show that benchmark equations such as the Korteweg--de Vries equation are intrinsically incompatible with the diamond update, while systems including the nonlinear Dirac and ``good'' Boussinesq equations admit stability regimes under mild timestep scaling. Extensive numerical experiments confirm the theoretical predictions and demonstrate the practical implications of the proposed framework.
Our results clarify fundamental limitations of diamond schemes and provide practical guidelines for their reliable application to new PDE models.

\end{abstract}
\begin{keyword}
    Geometric numerical integrator \sep multi-symplectic integrator \sep multi-symplectic diamond schemes \sep numerical stability
\end{keyword}

\maketitle

\section{Introduction}
In this paper, we consider the numerical computations of the multi-symplectic partial differential equations (PDEs) of the form:
\begin{equation}\label{eq:multi-symplectic}
    K z_t + L z_x = \nabla S(z),
\end{equation}
where $z = z(x, t)$ is a vector of dimension $d$, $K$ and $L$ are $d \times d$ skew-symmetric matrices, and $S(z)$ is a scalar function~\cite{bridges1997multi,BRIDGES_1997}.
Many physical systems, particularly integrable systems, can be written as multi-symplectic PDEs. For example, the Klein--Gordon equation $u_{tt} - u_{xx} = f(x)$ can be written in the form of Eq.~\eqref{eq:multi-symplectic} by introducing auxiliary variables $v = u_t$, $w = u_x$ and setting
\begin{align} \label{eq:LKG}
K = \begin{pmatrix}
    0 & -1 & 0 \\
    1 & 0 & 0 \\
    0 & 0 & 0
\end{pmatrix}, \quad
L = \begin{pmatrix}
    0 & 0 & 1 \\
    0 & 0 & 0 \\
    -1 & 0 & 0
\end{pmatrix}, \quad
S(z) = V(u) + (v^2 - w^2) / 2,
\end{align}
where $V(u)$ satisfies $V'(u) = -f(u)$. Other equations, such as the Korteweg--de Vries equation, the nonlinear Schr\"{o}dinger equation, the nonlinear Dirac equation, and the Boussinesq equation, can also be written in the form of Eq.~\eqref{eq:multi-symplectic}.
A strong feature of the multi-symplectic PDEs is that they satisfy the local conservation law of symplecticness:
\begin{equation}\label{eq:local_conservation}
    \pdv{t} \left( \frac{1}{2} \dd{z} \wedge K \dd{z}\right) + \pdv{x} \left( \frac{1}{2} \dd{z} \wedge L \dd{z}\right) = 0.
\end{equation}
From this, the global conservation laws of energy and momentum follow under appropriate boundary conditions.
In this sense, the local law provides more detailed information about the geometric structure of the PDE and is thus expected to be better respected in numerical computations rather than simply preserving global first integrals.

Multi-symplectic numerical integrators introduced in the seminal paper by Bridges and Reich~\cite{BR01} preserve the geometric structure in a discrete setting.
There, an implicit midpoint rule type integrator in both time and space directions was introduced, which is now called the Preissmann box scheme.
Later, Moore and Reich~\cite{moore2003backward} introduced an explicit symplectic Euler-type scheme, the so-called Euler box scheme.
These two numerical methods involve a trade-off between their respective advantages and disadvantages, as do the standard explicit Euler method (explicit but potentially unstable) and the implicit midpoint rule (potentially stable but fully-implicit), and should be chosen on a case-by-case basis depending on the PDE to be solved.
A partial list of the studies on specific PDEs is:  the KdV equation~\cite{AM05,WW13, ZQ00}, the nonlinear Schr\"{o}dinger equation~\cite{ZYY13,CX13,SQ03,HLMA06,CZT02}, the nonlinear Dirac equation~\cite{HL06}, and the Boussinesq equation~\cite{Chen05}.

In 2015, McLachlan and Wilkins proposed an intriguing new class of multi-symplectic integrators, the diamond schemes~\cite{Diamond}.
They utilize unconventional diamond-shaped mesh in space and time, and as a consequence, the schemes are {\em implicit only locally}; although the schemes are implicit, the nonlinear equations to be solved are separated in each diamond (i.e., a cell), and accordingly, the required computational complexity is $O(N)$ with $N$ the number of meshes in spatial direction (the computations could be even parallelized).
If not as fast as the explicit Euler box scheme, the diamond schemes should be significantly faster than the fully-implicit Preissmann box scheme.
Nevertheless, the diamond scheme retains a discrete version of the local symplectic law~\eqref{eq:local_conservation}, and thus it is expected to be an interesting third option that falls somewhere between the Euler box scheme and the Preissmann box scheme.
In addition to the basic diamond scheme, McLachlan and Wilkins also proposed higher-order versions based on the Runge--Kutta method.
They also showed that the diamond schemes work remarkably well for the nonlinear Klein--Gordon equation.
This study was followed by Marsland et al.~\cite{MMW20}, which analyzed aspects of the schemes, including parallelization and boundary treatments.
There, the nonlinear Klein--Gordon equation was considered as a benchmark again.

Despite the appealing properties and success with the nonlinear Klein--Gordon equation, the behavior of diamond schemes for other multi-symplectic PDEs has remained largely unexplored.
One reason might be that although previous studies implicitly assumed that solvability and stability behave similarly to other multi-symplectic schemes, which turns out not to be the case.
This motivates us to theoretically investigate the stability of the diamond schemes, but because of the spatiotemporal mesh's special structure, standard stability analyses for numerical differential equations do not apply.
The standard Dahlquist test is inadequate, since multi-symplectic PDEs are essentially systems (recall that the test is not meaningful for symplectic integrators, which assume a Hamiltonian system).
The von Neumann test cannot be used, since the diamond-shaped mesh is obliquely aligned and lacks the translational invariance the test assumes.

In view of this, the aim of this paper is to propose a complete procedure by which we can identify the stability (and consistency) of the diamond scheme applied to general multi-symplectic PDEs.
Through some illustrative examples, we show that there are three cases: (i) the diamond scheme is ``unsolvable'' (whose precise meaning will be defined later), (ii) the scheme is unconditionally unstable, and (iii) conditionally stable.
In the former two cases, we give up using the diamond scheme.
The third case is the happy case, if the condition for stability is not too harsh.
The proposed procedure consists of Steps 1--3, corresponding to the three cases above; it filters out the cases (i) and (ii) in this order, and finally gives a stability condition such that the diamond scheme works for the target PDE.
Fig.~\ref{fig:flowchart_stability_analysis} shows the flowchart of the procedure.

The overall stability analysis process differs significantly from that for standard numerical integrators; accordingly, it is helpful to briefly describe the essence of the three steps here.
In Steps 1 and 2, we consider the simple diamond scheme, occasionally mentioning extensions to higher-order versions when appropriate.
This is because it is unlikely that higher-order versions work better when the simple diamond scheme fails.
\begin{itemize}
    \item {\bf Step 1: Detecting structural inconsistency.}
     We first detect any structural inconsistency of the diamond scheme for the target PDE. This is done utilizing a tool from system analysis, the Dulmage--Mendelsohn decomposition of bipartite graphs, expressing the equation-unknown correspondence. We show that if the scheme is structurally inconsistent, it is hard to properly define a solution, and, in this sense, the scheme is ``unsolvable.''
     Surprisingly, the Korteweg--de Vries equation, a benchmark PDE for geometric numerical integrators, yields a structurally inconsistent multi-symplectic form, and the diamond scheme cannot be used.
    \item {\bf Step 2: Estimating the stability condition.}
    If the target PDE passes the first test, we proceed to the second test, where we estimate the stability condition. This test consists of constructing a simple weighted graph describing the error propagation and extracting cycles in the graph. The cycle weights provide a necessary condition for stability.
    It turns out in some PDEs the diamond scheme is ``unconditionally unstable,'' i.e., unstable for the stability constraint $\Delta t = O(\Delta x^s)$ for any positive $s$ ($\Delta t$ and $\Delta x$ are the time and spatial mesh sizes.
    They will be more precisely defined later.)
\end{itemize}
Although rigorous justification of these steps requires tools from system analysis and graph theory, we also provide  at the end of Section~\ref{sec:step2} a shortcut method to determine the conclusion without delving into the mathematical details.
This shortcut can be invoked in a pen-and-paper style only from the concrete form of the target multi-symplectic PDE.
We believe this is enough for most readers who hope to apply the diamond scheme to a new PDE.

Step 3 is optional, used only when we need the precise stability condition.
\begin{itemize}
    \item {\bf Step 3: Identifying precise stability condition by eigenvalue analysis.} The stability condition derived in Step 2 is a necessary condition, and generally not sufficient. If a sufficient condition is required, we perform an eigenvalue analysis in Step 3. This step requires complex matrix construction and the numerical computation of eigenvalues of small matrices.
\end{itemize}

Table~\ref{tb:stabilities_for_MS_PDEs} summarizes the result of the stability analysis of the diamond schemes for various multi-symplectic PDEs.
We have newly found that the diamond scheme for the Dirac equation is conditionally stable under the mild stability constraint $\Delta t = O(\Delta x)$.
We also found that for the ``good'' Boussinesq equation and the nonlinear Schr\"odinger equation, the diamond scheme is stable under $\Delta t = O(\Delta x^3)$.

The rest of this paper is organized as follows.
Section~\ref{sec:diamond} briefly summarizes the simple diamond scheme and the higher-order versions.
Sections~\ref{sec:step1}--\ref{sec:eigen} explain the procedure (Step 1--3).
In Section~\ref{sec:numerical}, we show some numerical examples for the PDEs newly found to be stable (with the diamond scheme).
Section~\ref{sec:conclusion} is for the conclusions and further comments.

\begin{figure}
    \begin{center}
        \includegraphics[width=125mm]{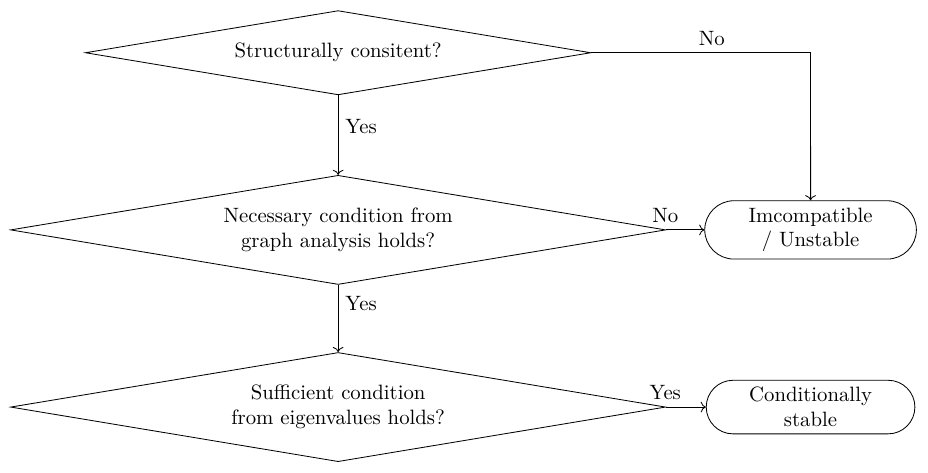}
    \end{center}
    \caption{Flowchart of the stability analysis process for the diamond schemes}
    \label{fig:flowchart_stability_analysis}
\end{figure}

\begin{table}[htbp]
    \centering
    \caption{Stabilities of the diamond schemes for multi-symplectic PDEs}
    \begin{tabular}{c|c|l}
        \hline
        Consistency/Stability & Situation & \multicolumn{1}{c}{PDEs} \\ \hline
        Structurally & \multirow{2}{*}{Over/underdetermined} & Advection, KdV, Camassa-Holm,\\
        inconsistent &  & BBM, Hunter--Saxton \\ \cline{1-3}
        \multirow{2}{*}{Unstable} & \multirow{2}{*}{Unconditionally unstable} & $u_{tx} = f(u)$, Ostrovsky,\\
        & & Improved Boussinesq\\ \hline

        \multirow{4}{*}{Conditionally stable}
        & \multirow{2}{*}{Stable with $\Delta t = O(\Delta x)$} & $u_{tt} - u_{xx} = f(u)$,\\
        &  & Dirac\\ \cline{2-3}
        & \multirow{2}{*}{Stable with $\Delta t = O(\Delta x^3)$} & ``Good'' Boussinesq,\\
        & & Nonlinear Schr\"{o}dinger\\ \hline
    \end{tabular}
    \label{tb:stabilities_for_MS_PDEs}
\end{table}

\section{A brief introduction of diamond schemes} \label{sec:diamond}
\subsection{The simple diamond scheme}
In what follows, we consider the multi-symplectic PDE~\eqref{eq:multi-symplectic} on a spatial domain $[a, b]$ and a temporal domain $[0, T]$ with periodic boundary conditions $z(a, \cdot) = z(b, \cdot)$ and initial conditions $z(x, 0) = z_0(x)$. We begin by briefly reviewing the construction of the simple diamond scheme. In the scheme, the spatial domain $[a, b]$ is divided into $N$ subintervals of length $\Delta x = (b-a) / N$ and the temporal domain $[0, T]$ is divided into $N_t$ subintervals of length $\Delta t = T / N_t$. For the simple diamond scheme, a diamond-shaped mesh is constructed in space and in time, as shown in Fig.~\ref{fig:diamond_mesh}. For the cases where we need to distinguish each diamond and its vertices, let spatial and time indices to be $n\in \{0, 1/2, 1, 3/2, \ldots, N\}$ and $i\in \{0, 1/2, 1, 3/2, \ldots\}$, and consider a collection of the points $(n \Delta x, i\Delta t)$'s.
We say ``vertex $(n,i)$'' when it is at $(n\Delta x, i\Delta t)$, and
also ``diamond $(n,i)$'' when its center is at $(n\Delta x, i\Delta t)$.

\subsubsection{Updating a single diamond}
For every single diamond (Fig.~\ref{fig:single_diamond}), a numerical solution of $z$ on the top corner is updated using known numerical solutions on the left, right, and bottom corners, by solving the following nonlinear equation:
\begin{equation}\label{eq:diamond_update}
    K \left(\frac{\boxed{z^\mathrm{t}} - z^\mathrm{b}}{\Delta t}\right) + L \left(\frac{z_\mathrm{r} - z_\mathrm{l}}{\Delta x}\right) = \nabla S(\boxed{z_0}), \qquad
    \text{where } z_0 := \frac{\boxed{\dsv{z}{t}} + \dsv{z}{b} + \dsh{z}{l} + \dsh{z}{r}}{4},
\end{equation}
where $z^\mathrm{t}, z^\mathrm{b}, z_\mathrm{l}, z_\mathrm{r}$ represent numerical solutions of $z$ at the top, bottom, left, and right corners of the diamond, respectively\footnote{Note that $z_t$ in the continuous PDE~\eqref{eq:multi-symplectic} and $\dsv{z}{t}$ in the discrete diamond scheme~\eqref{eq:diamond_update} are different.}.
The notation \boxed{\rm box} has no mathematical meaning; it is just for presentation and shows that the variable includes the unknown $\dsv{z}{t}$.
This is useful to identify the order of calculation in the subsequent analysis. We use this notation throughout this paper.

\begin{figure}[htbp]
    \begin{minipage}[t]{.48\linewidth}
        \begin{center}
            \includegraphics{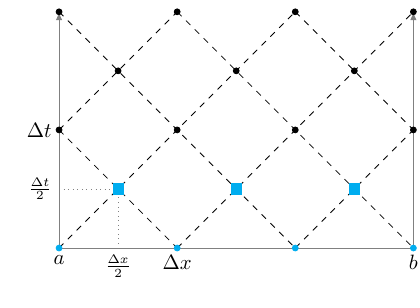}
        \end{center}
        \caption{Diamond-shaped mesh}
        \label{fig:diamond_mesh}
    \end{minipage}
    \begin{minipage}[t]{.48\linewidth}
        \begin{center}
            \includegraphics{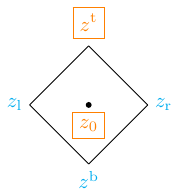}
        \end{center}
        \caption{A single diamond}
        \label{fig:single_diamond}
    \end{minipage}
\end{figure}

\subsubsection{Overall updating procedure}
The overall updating procedure of the simple diamond scheme is as follows: At first, the initial values of $z_{i, 0}$ are computed by the initial conditions. The first half-integer time solutions (the cyan square points in Fig.~\ref{fig:diamond_mesh}) are computed using some typical schemes such as the forward Euler method or the Preissmann box scheme. At each time step, numerical solutions for each diamond are computed by solving Eq.~\eqref{eq:diamond_update} independently. This updating procedure is repeated until the final time $T$ (i.e., $i = N_t$). The implicitness of the scheme is localized within each diamond. This local implicitness is an extremely significant advantage of the diamond scheme.

\subsection{The diamond scheme: high-order schemes using Runge--Kutta collocation methods}
In this subsection, the schemes referred to as \textit{the diamond schemes}, in the original paper written by McLachlan and Wilkins~\cite{Diamond}, are introduced. The diamond schemes are high-order extensions of the simple diamond scheme by combining it with the multi-symplectic Runge--Kutta collocation methods proposed by Reich~\cite{Reich}. In what follows, let $r$ denote the number of stages of the Runge--Kutta methods, and $A, b, c$ denote the coefficients of the Runge--Kutta methods. The elements of a matrix are denoted by $a_{ij}$, and those of a vector by $b_i$.

To apply the collocation methods to each diamond, first, the diamond has to be transformed into ordinary coordinate space. By the linear transform
\[
R:\quad \tilde{x} = \frac{1}{\Delta x} x + \frac{1}{\Delta t} t, \quad \tilde{t} = -\frac{1}{\Delta x} x + \frac{1}{\Delta t} t,
\]
where $(x, t)$ denotes the original coordinate in the diamond-shaped mesh, a unit square is obtained.
Then Eq.~\eqref{eq:multi-symplectic} becomes
\[
\widetilde{K} \tilde{z}_{\tilde{t}} + \widetilde{L} \tilde{z}_{\tilde{x}} = \nabla S(\tilde{z}),
\]
where
\[
\widetilde{K} = \frac{1}{\Delta t} K - \frac{1}{\Delta x} L, \qquad \widetilde{L} = \frac{1}{\Delta t} K + \frac{1}{\Delta x} L,
\]
and $\tilde{z} (\tilde{t}, \tilde{x}) = z(t, x)$.

The update in the transformed space is as follows.
A transformed diamond consists of $r \times r$ internal stage points $Z_i^j$ determined by the Runge--Kutta coefficients, as shown in Fig.~\ref{fig:RK_single_diamond}. Corresponding to the internal stage points $Z_i^j$, we prepare space and time variables $X_i^j$ and $T_i^j$, which are the approximations of $z_x$ and $z_t$, respectively.
\begin{figure}[htbp]
    \begin{center}
        \includegraphics{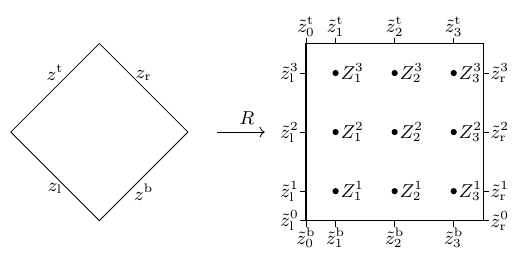}
    \end{center}
    \caption{A transformed diamond for the (higher-order) diamond schemes using Runge--Kutta collocation methods}
    \label{fig:RK_single_diamond}
\end{figure}

The Runge--Kutta discretization of this transformed diamond is as follows.
\begin{equation}\label{eq:RK_update}
    Z_i^j = \tilde{z}_\mathrm{l}^j + \sum_{k = 1}^r a_{ik} X_k^j,
    \qquad
    Z_i^j = \tilde{z}^\mathrm{b}_i + \sum_{k = 1}^r a_{jk} T_i^k,
    \qquad
    \nabla S(Z_i^j) = \widetilde{K} T_i^j + \widetilde{L} X_i^j.
\end{equation}
After solving these equations, we obtain the new numerical solutions over the upper sides of the original diamond, which are $z^\mathrm{t}$ and $z_\mathrm{r}$ on the left side of Fig.~\ref{fig:RK_single_diamond}, as
\begin{equation}\label{eq:RK_output}
    \tilde{z}_\mathrm{r}^i = \tilde{z}_\mathrm{l}^i + \sum_{k = 1}^r b_k X_k^i, \qquad
    \tilde{z}^\mathrm{t}_i = \tilde{z}^\mathrm{b}_i + \sum_{k = 1}^r b_k T_i^k.
\end{equation}
In the diamond scheme, this process is repeated for each diamond and for each time step to obtain numerical solutions over the entire diamond mesh.

\subsection{Properties of diamond schemes}
The simple diamond scheme preserves the multi-symplectic conservation laws of the original PDE~\eqref{eq:multi-symplectic}. More precisely, the scheme reproduces the local conservation law Eq.~\eqref{eq:local_conservation} in the discrete form as
\begin{equation*}\label{eq:local_conservation_discrete}
    \begin{split}
        &\phantom{{}+{}} \frac{1}{4\Delta t} ((\dd{z_\mathrm{l}} + \dd{z^\mathrm{t}} + \dd{z_\mathrm{r}}) \wedge K \dd{z^\mathrm{t}} - (\dd{z_\mathrm{l}} + \dd{z^\mathrm{b}} + \dd{z_\mathrm{r}}) \wedge K \dd{z^\mathrm{b}}) \\
        &+ \frac{1}{4\Delta x} ((\dd{z^\mathrm{t}} + \dd{z_\mathrm{r}} + \dd{z^\mathrm{b}}) \wedge L \dd{z_\mathrm{r}} - (\dd{z^\mathrm{t}} + \dd{z_\mathrm{l}} + \dd{z^\mathrm{b}}) \wedge L \dd{z_\mathrm{l}} ) = 0.
    \end{split}
\end{equation*}

The simple diamond scheme has a quadratic order of accuracy in time and space for the simple Klein--Gordon equation~\cite{Diamond}. In addition, the (high order) diamond schemes with Runge--Kutta collocation methods seem to have $r$-th order when $r$ is even and $(r+1)$-th order when $r$ is odd~\cite{Diamond}. Here, we stress that accuracy and stability are different properties of numerical schemes, and so there should be different treatments for the stability of the diamond schemes. In this paper, we discuss the stability of the diamond schemes in detail.

\section{Step 1: Detecting structural inconsistency of diamond schemes}
\label{sec:step1}

The first step of our analysis is to decide whether a diamond scheme admits a \emph{structurally consistent} local update. If not, the scheme may not be in general uniquely and locally solved.
In many traditional time-marching schemes, solvability issues rarely arise: even for nonlinear schemes, one typically obtains a (locally) unique solution at each time step for sufficiently small $\Delta t$.
In contrast, the diamond scheme may fail to provide a proper way to update the unknowns appearing in $\dsv{z}{t}$.
This failure is often structural, in the sense that the inconsistency arises from the algebraic structure of the diamond scheme rather than the actual coefficient values.
We show that such an inconsistency can be efficiently detected by the standard technique in system analysis, the Dulmage--Mendelsohn decomposition of a bipartite graph representing the relation between equations and unknowns.

\subsection{Motivating examples}

\begingroup
\renewcommand{\theequation}{A\arabic{equation}}
\setcounter{equation}{0}
Before presenting the rigorous system analysis approach, let us observe some motivating examples.
Let us first consider the linear advection equation
\[
u_t + u_x = 0.
\]
It can be written in multi-symplectic form by setting $z=(\phi,u,w)$ (cf.~\cite{Leimkuhler}):
\begin{align}
\boxed{u_t} + w_x &= 0, \label{eq:AE1}\\
-\boxed{\phi_t} &= 2\boxed{u}-\boxed{w}, \label{eq:AE2}\\
-\phi_x &= -\boxed{u}. \label{eq:AE3}
\end{align}
As described, the notation \boxed{\rm box} means that the terms include some unknown variables in $\dsv{z}{t}$ after the diamond scheme~\eqref{eq:diamond_update} is applied.
The equations~\eqref{eq:AE1} and \eqref{eq:AE3} have boxed terms only in terms of $u$, which means that both update the same variable $\dsv{u}{t}$.
Unless the two (discrete) formulae coincide, the local system is overdetermined.
\endgroup

\begingroup
\renewcommand{\theequation}{HS\arabic{equation}}
\setcounter{equation}{0}
Next, let us consider the Hunter--Saxton equation:
\[
u_{xxt} + 2u_xu_{xx} + uu_{xxx} = 0,
\]
which admits several multi-symplectic formulations~\cite{MCFM17}.
Setting $z=(u,\phi,w,v,\eta)$, we find the form:
\begin{align}
-\frac12\boxed{\eta_t} - v_x &= -\boxed{w} - \frac12\boxed{\eta}^2, \label{eq:HS1_1_new}\\
w_x &= 0, \label{eq:HS1_2_new}\\
-\phi_x &= -\boxed{u}, \label{eq:HS1_3_new}\\
u_x &= \boxed{\eta}, \label{eq:HS1_4_new}\\
\frac12\boxed{u_t} &= -\boxed{u\eta} + \boxed{v}. \label{eq:HS1_5_new}
\end{align}
In this formulation, the variable $\phi$ never appears in the boxed terms,
and accordingly, there is no update formula for $\dsv{\phi}{t}$.
In this case, the system is underdetermined.
\endgroup

In both cases, the multi-symplectic forms exhibit structural inconsistencies, rendering the simple diamond scheme unsuitable.
In the following subsection, we will show that such structural inconsistency can be detected using a standard tool from structural system theory: the Dulmage--Mendelsohn decomposition of bipartite graphs.

\subsection{Equation--unknown bipartite graph and DM decomposition}

Let us consider the multi-symplectic PDE $K z_t + L z_x = \nabla S(z)$, and the local system of the simple diamond scheme~\eqref{eq:diamond_update} on a single diamond.
Let $\mathcal{E}$ be the set of scalar equations in this local system,
and $\mathcal{U}$ be the set of scalar unknowns to be determined on that diamond.
We define the bipartite graph:
\[
\gbi = (\mathcal{E} \sqcup \mathcal{U}, E)
\]
by connecting $e \in\mathcal{E}$ and $u \in\mathcal{U}$ when the unknown $u$ appears in the equation $e$.
The symbol $\sqcup$ is the disjoint union, and $E$ is the collection of the edges.

The Dulmage--Mendelsohn (DM) decomposition of $\gbi$ partitions the vertex sets into blocks (see, e.g., Murota~\cite{murota2010matrices}):
\[
(\mathcal{E}_0,\mathcal{U}_0),\ (\mathcal{E}_1,\mathcal{U}_1),\ \dots,\ (\mathcal{E}_h,\mathcal{U}_h),\ (\mathcal{E}_\infty,\mathcal{U}_\infty),
\]
where $h$ is an integer (determined as the output of the DM decomposition algorithm).
The decomposition reveals the following structural information on the system.

\begin{itemize}
\item The blocks $(\mathcal{E}_k,\mathcal{U}_k)$ ($k=1,\dots,h$) are \emph{well-determined} blocks in the sense that they admit perfect matching between equations and unknowns within each block.
Moreover, the decomposition induces a partial order among these blocks, providing a natural ``solve-from-upstream'' computational order at the structural level.
\item The block $(\mathcal{E}_0, \mathcal{U}_0)$ corresponds to an \emph{overdetermined} part: there are equations that cannot be matched to unknowns.
\item The block $(\mathcal{E}_\infty, \mathcal{U}_\infty)$ corresponds to an \emph{underdetermined} part: there are unknowns that cannot be matched to equations.
\end{itemize}

Using this decomposition, we formally define structural inconsistency as follows.
Unlike standard time-marching schemes, the diamond update may fail even before stability becomes an issue.
\begin{dfn}[Structural inconsistency]
\label{def:struct_inconsistency}
We say that the (simple) diamond scheme is \emph{structurally inconsistent} on the target PDE if the DM decomposition of $\gbi$ has non-zero blocks $(\mathcal{E}_0, \mathcal{U}_0)$ or $(\mathcal{E}_\infty, \mathcal{U}_\infty)$.
\end{dfn}

Figs.~\ref{fig:advection_DM} and~\ref{fig:HS1_DM} show the decomposed graphs for the previous two examples.
For the advection equation (Fig.~\ref{fig:advection_DM}), there is no well-determined block, and two equations (and accordingly their corresponding discrete formulae) are categorized into the overdetermined block $\mathcal{E}_0, \mathcal{U}_0$.
The last equation (formula) is categorized into the underdetermined block $\mathcal{E}_\infty, \mathcal{U}_\infty$, which is a consequence of the fact that the former two equations are overdetermined.
For the Hunter--Saxton case (Fig.~\ref{fig:HS1_DM}), one equation belongs to the overdetermined block, and one to the underdetermined block.
Note that the reason of \eqref{eq:HS1_2_new} $w_x=0$ being categorized as overdetermined is that it results in the discrete formula $(\dsh{w}{r}-\dsh{w}{l})/\Delta x = 0$ on the diamond, which casts an additional constraint on $\dsh{w}{r}$ and $\dsh{w}{l}$, which should have been already calculated independently in the lower left and lower right diamonds.
In these two examples, the above information agrees with the observations we had by sight.

\begin{figure}[htbp]
\centering

\includegraphics{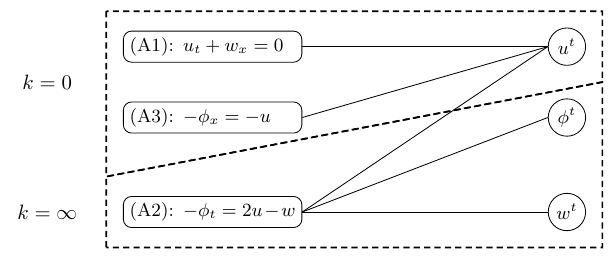}
\caption{Dulmage--Mendelsohn decomposition of the bipartite graph for the advection equation: equations vs. unknowns \((\phi^{t},u^{t},w^{t})\).}\label{fig:advection_DM}
\end{figure}

\begin{figure}[htbp]
\centering
\includegraphics{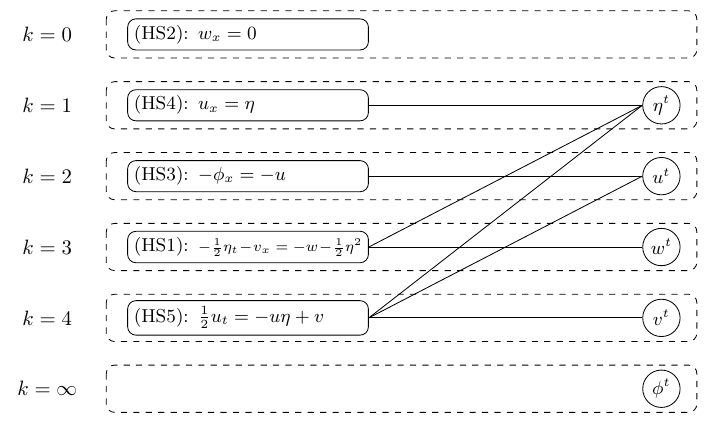}
\caption{Dulmage--Mendelsohn decomposition of the bipartite graph for the Hunter--Saxton equation (first multi-symplectic form): equations vs. unknowns \((u^{t},\phi^{t},w^{t},v^{t},\eta^{t})\). }\label{fig:HS1_DM}
\end{figure}

The next proposition shows that when the system is structurally inconsistent, the simple diamond scheme is unsolvable.
We show this for linear PDEs.
Let us consider a linear multi-symplectic PDE where $\nabla S(z)=Pz$ with a constant matrix $P$.
Then the simple diamond scheme~\eqref{eq:diamond_update} reads
\[
K \frac{\dsv{z}{t} - \dsv{z}{b}}{\Delta t} + L \frac{\dsh{z}{r} - \dsh{z}{l}}{\Delta x} = P \left(\frac{\dsv{z}{t} + \dsv{z}{b} + \dsh{z}{r} + \dsh{z}{l}}{4}\right)\quad
\Leftrightarrow \quad
\left( \frac{K}{\Delta t} - \frac{P}{4}\right)\dsv{z}{t} = \frac{K}{\Delta t}\dsv{z}{b} + \frac{L}{\Delta x}(\dsh{z}{r} - \dsh{z}{l}) + P \left(\frac{\dsv{z}{b} + \dsh{z}{r} + \dsh{z}{l}}{4}\right).
\]

\begin{prop}[Structural inconsistency implies singular coefficient matrix]
\label{prop:DM_implies_singular}
Consider a linear multi-symplectic PDE, and suppose the simple diamond scheme is structurally inconsistent on it.
Then the coefficient matrix $(K/\Delta t-P/4)$ is singular.
\end{prop}

\begin{proof}
Recall some facts from system analysis (see, for example,~\cite[Chap.~2]{murota2010matrices}): (i) A DM decomposition of a bipartite graph $\gbi$ has a nonempty overdetermined block ($\mathcal{E}_0, \mathcal{U}_0$) or an underdetermined block ($\mathcal{E}_\infty, \mathcal{U}_\infty$) if and only if $\gbi$ does not admit a perfect matching,
and (ii) the rank of the matrix whose non-zero elements are expressed by $\gbi$ is equal to or less than the size of the maximum matching in $\gbi$.

From the assumption, the bipartite graph $\gbi$ for the target multi-symplectic PDE does not have a perfect matching, and thus the rank of the matrix  $({K}/{\Delta t} - {P}/{4})$ is strictly less than the problem size, regardless of the value of $\Delta t$. This means that the coefficient matrix is singular.
\end{proof}

Some additional comments are in order.
First, in practice, the DM decomposition for a specific PDE can be easily found by sight, without relying on computational algorithms. We will show it at the end of Step 2.
Second, we show further examples of unsolvable PDEs in~\ref{app:examples}. It is remarkable that even the Korteweg--de Vries equation, the benchmark PDE for various geometric numerical integrators, is structurally inconsistent with the diamond scheme.
Third, when the simple diamond scheme is structurally inconsistent, the higher-order diamond schemes using the Runge--Kutta collocation methods are singular (\ref{app:high-order}).
Fourth, the main aim of the system analysis is to filter out structurally inconsistent cases at the outset; we do not consider the actual (unique) solvability of the well-determined blocks here.
When the PDE (and accordingly the scheme) is nonlinear, the unique existence of (numerical) solutions should be shown independently.
Even when the PDE is linear, the system may be singular depending on the numerical values of the coefficients.
These should be checked for each specific PDE.
Fifth, one may nonetheless think that we may determine the solution by supplementing the rank deficiency with additional compatibility conditions, even when the coefficient matrix is singular.
This does not seem like a good idea for several reasons.
The primary obstacle is that, in doing so, all the update formulae of the diamond scheme become globally coupled, thereby losing its greatest advantage.
We illustrate this in the next proposition.
Its detail shall be given in~\ref{app:enforce}.

\begin{prop}
\label{prop:singularity_implies_fully_implicit}
Consider a linear multi-symplectic PDE, and suppose the coefficient matrix ${K}/{\Delta t} - {P}/{4}$ in the simple diamond scheme is singular.
If we introduce an additional compatibility constraint such that the scheme admits a unique solution, then the diamond scheme is globally coupled.
\end{prop}
\begin{proof}
From the singularity of the coefficient matrix, there exists a vector $v\neq 0$ in its kernel.
The scheme should be in the form
\[
\left( \frac{K}{\Delta t} - \frac{P}{4} \right) \dsv{z}{t} = b
\]
for some vector $b$ constructed by collecting all the terms except for $\dsv{z}{t}$ to the right hand side.
The vector should be in the range of ${K}/{\Delta t} - {P}/{4}$, which demands the compatibility condition $v^\top b = 0$.
Recall that $b$ includes $\dsh{z}{l}$ and $\dsh{z}{r}$ (unless otherwise the multi-symplectic form has no $Lz_x$ part, which is meaningless as a PDE), and accordingly the condition defines a relation between them.
Since the left and right vertices are shared in adjacent diamonds, the condition globally couples the diamond schemes.
\end{proof}

\section{Step 2: Estimating the (in)stability } \label{sec:step2}

If a PDE passes the first test (Step 1), we proceed to the second test, which estimates the stability of the diamond scheme.
The main aim of this step is to rule out extreme cases in which the scheme is unconditionally unstable.
As usual, we only consider linear stability analysis; we assume that the target PDEs are linear or linearized beforehand.

\subsection{Error propagation: Observation in the wave equation} \label{subsec:wave}

Let us first observe how errors are propagated in the scheme~\eqref{eq:diamond_update}, taking the wave equation $u_{tt} - u_{xx} = 0$ as an example.
As shown in Section 1, the wave equation can be written in the multi-symplectic form by setting $z = (u \; v \; w)$,
\[
\begin{pmatrix}
    0 & -1 & 0 \\
    1 & 0 & 0 \\
    0 & 0 & 0
\end{pmatrix} z_t
+
\begin{pmatrix}
    0 & 0 & 1 \\
    0 & 0 & 0 \\
    -1 & 0 & 0
\end{pmatrix} z_x
= \nabla \left\{ \frac{1}{2} \left( v^2 - w^2 \right)\right\},
\]
which is expanded to
\begin{empheq}[left=\empheqlbrace]{align}
    -\boxed{v_t} + w_x &= 0, \label{eq:wave:c1}\\
    \boxed{u_t} &= \boxed{v}, \label{eq:wave:c2}\\
    -u_x &= -\boxed{w}. \label{eq:wave:c3}
\end{empheq}
In Step 1, we should have considered the corresponding DM decomposition and accordingly revealed the (partial) order of calculation: $\{v, w \}$, then $u$.

\begin{figure}[htbp]
\centering
\includegraphics{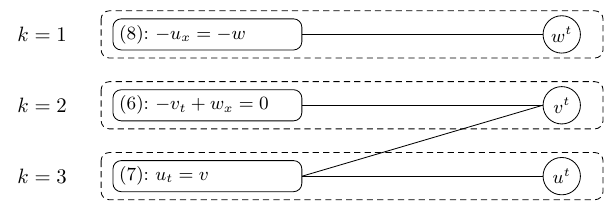}
\caption{Dulmage--Mendelsohn decomposition of the bipartite graph for the wave equation.}
\end{figure}

Let us see how the error included in some variable is propagated. To see this, let us write down the scheme~\eqref{eq:diamond_update} as follows.
\begin{align} \label{eq:wave:d}
    \left\{
    \begin{aligned}
        -\frac{\boxed{v^\mathrm{t}} - v^\mathrm{b}}{\Delta t} + \frac{w_\mathrm{r} - w_\mathrm{l}}{\Delta x} &= 0, \\
        \frac{\boxed{u^\mathrm{t}} - u^\mathrm{b}}{\Delta t} \phantom{{}+\frac{{w_\mathrm{r}} - w_\mathrm{l}}{\Delta x}} &= \frac{\boxed{v^\mathrm{t}} + v^\mathrm{b} + v_\mathrm{r} + v_\mathrm{l}}{4}, \\
        - \frac{u_\mathrm{r} - u_\mathrm{l}}{\Delta x} &= -\frac{\boxed{w^\mathrm{t}} + w^\mathrm{b} + w_\mathrm{r} + w_\mathrm{l}}{4}.\\
    \end{aligned}
    \right.
\end{align}
Rearranging terms, we have
\begin{align}
    \boxed{v^\mathrm{t}} &= v^\mathrm{b} +\frac{\Delta t}{\Delta x} (w_\mathrm{r} - w_\mathrm{l}),
    & \left(\text{which means}\;\dsv{v}{t}\; \leftarrow \; \dsv{v}{b}, \frac{\Delta t}{\Delta x} \{\dsh{w}{r}, \dsh{w}{l}\}  \right),
    \label{eq:KG1} \\
    \boxed{w^\mathrm{t}} &= \frac{4}{\Delta x} (u_\mathrm{r} - u_\mathrm{l}) -(w^\mathrm{b} + w_\mathrm{r} + w_\mathrm{l}),
    &  \left(\dsv{w}{t}\; \leftarrow \; \frac{4}{\Delta x}\{\dsh{u}{r},\; \dsh{u}{l}\},\; \{\dsv{w}{b}, \dsh{w}{r}, \dsh{w}{l}\}  \right),
    \label{eq:KG3}\\
    \boxed{u^\mathrm{t}} &= u^\mathrm{b} + \frac{\Delta t}{4} (\boxed{v^\mathrm{t}} + v^\mathrm{b} + v_\mathrm{r} +v_\mathrm{l}),
    & \left(\dsv{u}{t}\; \leftarrow \; \dsv{u}{b},\: \frac{\Delta t}{4} \{\dsv{v}{t}, \dsv{v}{b}, \dsh{v}{r}, \dsh{v}{l}\}  \right).
    \label{eq:KG2} \\
\end{align}
The rightmost explanations in $(\cdot)$ describe how the error propagates.
In~\eqref{eq:KG1}, the error in $\dsv{v}{b}$ is passed as is to $\dsv{v}{t}$;
those in $\dsh{w}{r}, \dsh{w}{l}$ is scaled by $\Delta t/\Delta x$ and then passed to $\dsv{v}{t}$.
The other two equations are similar.
Thus, if in the diamond a perturbation is added to $\dsv{u}{b}$, it should propagate to $\dsv{u}{t}$ in the following mechanism.
\begin{enumerate}
\item Propagates to $\dsh{w}{l}$ (and $\dsh{w}{r}$, respectively):  they are computed in the lower left (right, respectively) adjacent diamond as $\dsv{w}{t}$ by~\eqref{eq:KG3}. There, the error is multiplied by $4/\Delta x$.
\item Propagates to $\dsv{v}{t}$: by~\eqref{eq:KG1}, the error is further multiplied by $\Delta t/\Delta x$.
\item Propagates to $\dsv{u}{t}$: by~\eqref{eq:KG2}, the error is multiplied by $\Delta t/4$.
\end{enumerate}
In total, the initial error scales as $(\Delta t/\Delta x)^2$, and if $\Delta t > \Delta x$, the error in $u$ will grow exponentially in future diamonds, meaning the scheme is unstable.

Since the main aim of Step 2 of the stability analysis is to filter out unconditionally unstable cases, we are not interested in the precise ratio of $\Delta t$ and $\Delta x$, and the above analysis can be simplified.
Suppose $\Delta t \sim (\Delta x)^s$, where $s$ is a positive real number.
Then the above amplification factors are
$4/\Delta x \sim (\Delta x)^{-1}$,
$\Delta t/\Delta x \sim (\Delta x)^{s-1}$, and
$\Delta t/4 \sim (\Delta x)^{s}$.
This means the overall amplification factor is $\sim (\Delta x)^{2s-2}$.
This analysis suggests that it is more convenient to focus on the exponent of $\Delta x$, i.e., $(-1)+(s-1)+(-1)=2s-2$.
We call these indices the {\em amplification index} of a single update or its concatenation.
In the present case, if $s<1$, the error will blow up exponentially, rendering the scheme unstable.

This suggests that we consider such a graph as in Fig.~\ref{fig:graph_errors_simple_KG}.
The nodes $u, v,$ and $w$ represent the variables, and the directed edges mean the directions of the computations and their amplification indices.
We call the graph the {\em error-propagation graph (reduced version)} (the meaning of ``reduced version'' shall become clear later), and we denote it by $\gred$.
In $\gred$, we see a cycle $u\rightarrow w \rightarrow v \rightarrow u$ with the weight $2s-2$.
In the unstable case $s<1$, the cycle is negative, and in this sense, {\em the negative cycle is equivalent to the instability of the diamond scheme}.
In the next subsection, we will rigorously establish this statement in general cases.

\begin{figure}[htbp]
    \begin{center}
        \includegraphics{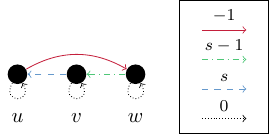}
    \end{center}
    \caption{(Reduced version of the) Error-propagation graph $\gred$ for the wave equation.}
    \label{fig:graph_errors_simple_KG}
\end{figure}

\subsection{Error propagation: General cases}

Let us consider the general PDE~\eqref{eq:multi-symplectic}
with the solution $z=(z_1\; z_2\; \ldots, z_d)$ ($d$ is the number of the variables),
and the corresponding diamond scheme~\eqref{eq:diamond_update}.
We assume the scheme is at least consistent (since inconsistent cases have already been eliminated in Step~1).

We provide a formal, general definition of the error-propagation graph (reduced version).
\begin{dfn}[Error-propagation graph (reduced version)]\label{def:error_graph}
    The {\em error-propagation graph (reduced version)} $\gred$ for a multi-symplectic PDE~\eqref{eq:multi-symplectic} and its simple diamond scheme~\eqref{eq:diamond_update} is the directed and weighted graph defined according to the rules:
    \begin{itemize}
    \item Nodes: Place a node for each variable  $z_i$ ($i\in \{1,\ldots,d\})$ in $z$.
    \item Edges: For each update formula (such as Eq.~\eqref{eq:KG1}--Eq.~\eqref{eq:KG2}) representing the target diamond scheme, place a directed edge from $z_i$ to $z_j$ ($i,j \in \{1,\ldots,d\}$)
    when $z_j$ is computed using $z_i$. The weight of the edge is the error amplification index of the error from $z_i$ to $z_j$.
    \end{itemize}
\end{dfn}

Below is our main result in this section.

\begin{thm}\label{thm:graph_stability}
    Consider a multi-symplectic PDE~\eqref{eq:multi-symplectic} and its simple diamond scheme~\eqref{eq:diamond_update} under the step size condition $\Delta t \sim (\Delta x)^s$ ($s>0$).
    The scheme is unstable if and only if the associated error-propagation graph contains a negative cycle for the choice of $s$.
    (If there is a negative cycle which persists for all $s>0$, then we say that the scheme is unconditionally unstable.)
\end{thm}

For its proof, we need some preparations.
Let us consider how an error in the variable $v_j$ ($j\in \{1,\ldots,d\})$ at some vertex $(n,i)$ expands. As time integration proceeds to $t=(i+1/2)\Delta t, (i+1)\Delta t, \ldots$, the error propagates to the vertices $(n\pm 1/2, i+1/2)$ (the error expands to the two vertices), $\{ (n, i+1), (n\pm 1, i+1)\}$ (three vertices), $\{ (n\pm 1, i+3/2), (n\pm 2, i+3/2)\}$ (four vertices), and so on, through variables one after another. Like the region of influence in hyperbolic partial differential equations, the region of influence of the error forms an inverted triangle shape.
We hope to identify the worst-case error propagation in that region and check whether it can go to infinity as $t\to\infty$ (indicating instability).
The difficulty is that, as time evolves, there are combinatorially many error-propagation routes.
However, we can limit ourselves to specific routes due to the symmetry of the diamond scheme.

The observations regarding route limitations are as follows. Mind that all the diamond employs the same update formulae.
\begin{itemize}
    \item {\bf Error propagation in one variable.}  The error in $v_j$ at the vertex $(n,i)$ will always propagate back to the same variable $v_j$ at some vertex $(n', i')$ with $i'>i$, since there are only finite number of variables. Thus, it is sufficient to focus on the error amplification in $v_j$ to identify unstable cases. The choice of $v_j$ is arbitrary.
    \item {\bf Symmetry in spatial direction.} The simple diamond scheme~\eqref{eq:diamond_update} includes $\dsh{z}{l}$ and $\dsh{z}{r}$ in a symmetric way (in terms of the amplification index). Thus, in the updates $\dsv{z}{b}\to \dsh{z}{l}$ and $\dsh{z}{r}$, the error expands with exactly the same amplification factor.
    And for the same reason, the variable $\dsv{z}{t}$ receives the scattered errors from $\dsh{z}{l}$ and $\dsh{z}{r}$ with the equal amplification index.
    (Note that we use the same notation $\dsv{z}{t}$ and so on for all the diamonds, omitting the precise time and spatial index.
    The update $\dsv{z}{b}\to \dsh{z}{l}$ (and $\dsh{z}{r}$, respectively) is actually done in the adjacent, lower left (and right) diamond.)
\end{itemize}

Let us more precisely define the terminology ``route'' and consider which routes we should consider. If the error in $v_j$ at the vertex $(n,i)$ propagates back to $v_j$ at $(n', i')$ and is amplified by the factor $(\Delta x)^q$ for some constant $q$, then we say ``there is a {\em route} from $(n,i)$ to $(n', i')$ in terms of $v_j$ with the amplification index $q$.''
The next lemma states that we can focus on a single route from $(n, i)$ to $(n, i')$ and its amplification index.
\begin{lem}
Suppose there is a route from the vertex $(n, i)$ to $(n', i')$ in terms of $v_j$ with the amplification index $q$. If $n$ and $n'$ are both integers or half-integers, then there exists a route from $(n, i)$ to $(n, i')$ in terms of $v_j$ with the same amplification index $q$. If $n$ and $n'$ are a set of integer and half-integer, then by doubling the original route, which becomes a route from $(n, i)$ to $(2n'-n, 2i' - i)$ with the index $2q$, we can find a route from $(n, i)$ to $(n, 2i'-i)$ with the index $2q$.
\end{lem}
\begin{proof}
    Since the error propagation is symmetric in space (in terms of the amplification index), at every update $\dsv{z}{b}$ to $\dsh{z}{l}$ and $\dsh{z}{r}$, we can choose one from them such that the spatial index stays closer to the original index $n$, without changing the amplification factor.
    This can be repeated until we reach the time index $i'$.
    If $n$ and $n'$ are both integers or half-integers, then this gives the desired route.

    If, for example, $n$ is an integer and $n'$ is a half-integer, by the above procedure, the route can only come back to the spatial index $n \pm (1/2)$.
    In that case, repeat the original route from $(n', i')$, which makes a new route from $(n,i)$ to $(2n'-n, 2i'-i)$ with the amplification index $2q$.
    By applying the above procedure, we find a route that comes to the spatial index $n$.
\end{proof}

This lemma implies that it is sufficient to consider the vertical cascade of single diamonds rather than the entire region of influence.
Furthermore, since error propagation along the left and right sides is identical in terms of amplification, we can fix the route to pass through, say, the left side of the diamonds.
With this reduction, we can forget about the spatial index and focus on error propagation in the time direction.

The observation suggests that we consider the following graph, which starts at the time level $t=i\Delta t$ and explicitly shows the full time-evolution structure.
An example for the case of the wave equation is given in Fig.~\ref{fig:graph_errors_WE_local}.
The graph is infinitely long in the time direction, and we only show the initial part (which is enough for the subsequent analysis). Note that the graph's structure is uniform in time since the update formulae are the same in the diamonds.
\begin{dfn}[Error-propagation graph (full version)] \label{def:tepg}
The {\em error-propagation graph (full version)} $\gful$ for a multi-symplectic PDE~\eqref{eq:multi-symplectic} and its simple diamond scheme~\eqref{eq:diamond_update} is the directed and weighted graph defined according to the rules:
    \begin{itemize}
    \item Nodes: Place a node for each variable  $z_l$ ($l\in \{1,\ldots,d\})$ in $z$ for each time $t=j\Delta t$ (j=i, i+1/2, i+1, i+3/2, \ldots).
    \item Edges: For each update formula (such as Eq.~\eqref{eq:KG1}--Eq.~\eqref{eq:KG2}) representing the target diamond scheme, place directed edges according to the following procedure. We start with the time level $t=i\Delta t$, and set $j:=i$ (the initial data for $j$).
    \begin{itemize}
    \item[(i)] For every calculation from the time level $j$ to $j+1/2$, draw a directed edge upward between the corresponding variables, with the corresponding amplification index.
    \item[(ii)] Do the same for the calculations from the time level $j$ to $j+1$.
    \item[(iii)] Setting $j:=j+1/2$, go back to (i) and repeat the procedure (until some designated time level).
    \end{itemize}
    \end{itemize}
\end{dfn}

   \begin{figure}[htbp]
    \begin{center}
    \includegraphics{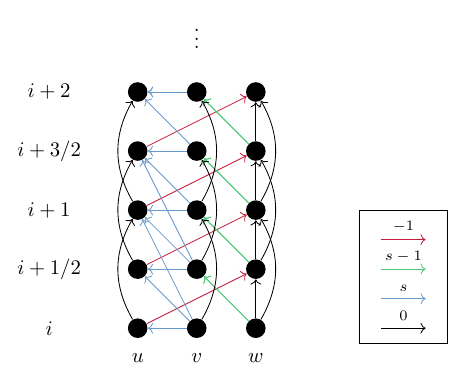}
    \caption{Error-propagation graph $\gful$ (full version; wave equation case).}
    \label{fig:graph_errors_WE_local}
    \end{center}
    \end{figure}

\begin{proof}[Proof of Thm.~\ref{thm:graph_stability}]
    Suppose we have a route from a vertex $(n, i)$ to $(n, i')$ for some $i'>i$ with the amplification factor $q(s)<0$ in terms of a variable, say, $v_1$, where $s>0$ comes from the assumption $\Delta t \sim (\Delta x)^s$. The notation $q(s)$ means the overall amplification index of the route depends on $s$.
    By repeating the same route, we see that the initial error in $v_1$ grows arbitrarily large, which means numerical instability.

    Construct the error-propagation graph $\gful$  for the case.
    There, the route with the amplification factor $q(s)$ appears as a directed path from $v_1$ at the time level $i$ to the same variable at some time level $i'>i$.
    Below, we call such a path a {\em returning path}.
    When its amplification factor $q(s)$ is negative, then we call it a {\em negative returning path.}
    (In case of the wave equation (Fig.~\ref{fig:graph_errors_WE_local}), $z_1=u$, and the path $u(i)\xrightarrow{-1}w(i+1/2)\xrightarrow{s-1}v(i+1)\xrightarrow{s}u(i+1)$ is a returning path, where the arguments of $u, w, v$ denote the time levels. The amplification factor of the returning path is $2s-2$, and it is negative for $s<1$.)

    We show that an error-propagation graph $\gful$ includes a negative returning path if and only if the associated reduced graph $\gred$ defined in Def.~\ref{def:error_graph} has a negative cycle.

    {\bf Sufficiency.} Suppose the full graph $\gful$ has a negative returning path.
    Projecting all the nodes and edges for the time levels $j>i$ to the time level $i$ (i.e., by compressing the graph $\gful$ in time direction), we find a cycle corresponding to the negative returning path in the reduced graph $\gred$. The weight on each edge remains the same, so the overall amplification factor stays the same; i.e., the cycle is negative.

    {\bf Necessity.} Suppose the reduced graph $\gred$ contains a negative cycle. We show that the cycle can be elevated to a corresponding negative path on the full graph $\gful$.
    Choose one variable in the negative cycle in $\gred$, say $z_j$, and suppose we start from $z_j$ at the time level $i$ in the full graph $\gful$.
    The negative cycle in $\gred$ identifies a unique transition from $z_j$ to the next variable $z_{j'}$. In the full graph $\gful$, there uniquely exists the corresponding directed edge from $z_j$ to $z_{j'}$; follow the edge. Repeat this until we return to $z_j$, which yields the desired negative path.
    (For the wave equation case, we take, for example, $z_j:=u$. Then in the negative cycle in Fig.~\ref{fig:graph_errors_simple_KG} demands us to find a path that evolves in the order $u\to w\to v\to u$. In the full graph $\gful$ Fig.~\ref{fig:graph_errors_WE_local}, this uniquely represents the path $u(i)\xrightarrow{-1}w(i+1/2)\xrightarrow{s-1}v(i+1)\xrightarrow{s}u(i+1)$.)
\end{proof}

Thm.~\ref{thm:graph_stability} shows that it is sufficient to consider $\gred$, which can be easily constructed from the multi-symplectic form of the target PDE.

Below, we show examples of PDEs for which the diamond scheme is unconditionally unstable.
It is an important caveat of the diamond multi-symplectic method that the schemes can be useless, however small $\Delta t$ is, even if the scheme itself is solvable and numerical solutions exist.
In the rest of this section, when we say error-propagation graph, we mean the reduced graph $\gred$.

\subsubsection{Mixed derivative type Klein--Gordon equation}
The mixed derivative type linear Klein--Gordon equation (also known as the linear Klein--Gordon under the light-cone coordinate) reads $u_{tx} = au$, where $a$ is a real constant. This can be written in the following multi-symplectic form by setting $z = (u\; v\; w)$.
\begin{equation*}
    \begin{split}
        \frac{1}{2}\boxed{v_t} + w_x &= -a\boxed{u},\\
        -\frac{1}{2}\boxed{u_t} &= \boxed{w}, \\
        -u_x &= \boxed{v}.
    \end{split}
\end{equation*}
The error-propagation graph for this equation is as follows.
This can be easily constructed by the observations described at the end of Section~\ref{subsec:wave}.
\begin{figure}[htbp]
    \begin{center}
        \includegraphics{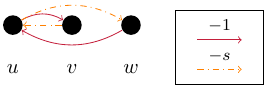}
    \end{center}
    \caption{The error propagation graph for the mixed-derivative type Klein--Gordon equation $u_{tx} = u$}
    \label{fig:graph_errors_mKG}
\end{figure}
This graph has \emph{two negative cycles with the index $-s-1$} (recall that we assume $\Delta t\sim (\Delta x)^s$, and $s$ is a positive number; $-s-1$ is negative for any positive $s$. This notice also applies to the examples below). Therefore, the simple diamond scheme applied to this equation is always unstable. This instability can be confirmed by numerical experiments. For example, we consider $u_{tx} = -\pi^2 u$ over the domain $[-1, 1]$. The initial value $u(0, x)$ was set to $\cos(\pi x)$, as shown in Fig.~\ref{fig:experiment_mLKG_init}, and $\Delta x = 0.05, \, \Delta t = 10^{-4}$. The time evolution of the simple diamond scheme after one time step is shown in Fig.~\ref{fig:experiment_mLKG_1step}. The numerical solution $u(\Delta t, x)$ is significantly different from the exact solution $u(\Delta t, x) = \cos(\pi (x + \Delta t))$, which indicates the instability of the scheme.

\begin{figure}[htbp]
    \begin{minipage}[t]{.48\linewidth}
        \begin{center}
            \includegraphics[width=75mm]{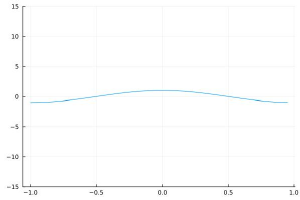}
            \caption{Initial condition $u(0, x) = \cos(\pi x)$ for the mixed-derivative type Klein--Gordon equation $u_{tx} = -\pi^2 u$.}
            \label{fig:experiment_mLKG_init}
        \end{center}
    \end{minipage}
    \hfill
    \begin{minipage}[t]{.48\linewidth}
        \begin{center}
            \includegraphics[width=75mm]{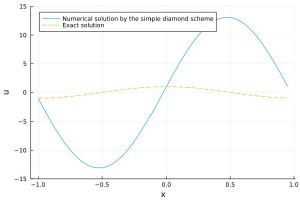}
            \caption{Exact solution and numerical solution after one time step obtained by the simple diamond scheme for $u_{tx} = -\pi^2 u$.}
            \label{fig:experiment_mLKG_1step}
        \end{center}
    \end{minipage}
\end{figure}

\subsubsection{Improved Boussinesq equation}
The improved Boussinesq equation $u_{tt} - u_{xx} = -u_{xxtt} + (u^2)_{xx}$ can be written in the following multi-symplectic form~\cite{CLB13} by setting $z = (u \; v \; n \; w \; p \; q)$:
\begin{equation*}
    \begin{split}
        \boxed{w_t} - p_x &= \boxed{u} + u^2,\\
        -\boxed{q_t} - n_x - w_x &= -\boxed{v}, \\
        v_x &= -\boxed{n}, \\
        -\boxed{u_t} + v_x &= 0, \\
        u_x &= \boxed{q}, \\
        v_t &= \boxed{p}.
    \end{split}
\end{equation*}

This is a nonlinear equation.
We assume small solutions, and ignore the nonlinear term $u^2$ (this is why  this term is not marked as unknown). The error propagation graph for this equation is shown in Fig.~\ref{fig:graph_errors_improved_Boussinesq}.
\begin{figure}[htbp]
    \begin{center}
        \includegraphics{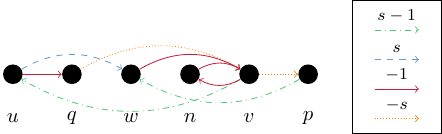}
    \end{center}
    \caption{The error propagation graph for the improved Boussinesq equation $u_{tt}-u_{xx}=-u_{xxtt}+(u^2)_{xx}$}
    \label{fig:graph_errors_improved_Boussinesq}
\end{figure}
This graph has two negative cycles: the cycle between $n$ and $v$, and the cycle $w\to v\to p\to w$, whose indices are $-2$.
These variables will quickly become large even if we start from an initial data set small enough. This means the simple diamond scheme is not appropriate for this equation.

\subsubsection{Ostrovsky equation}
The Ostrovsky equation $u_{tx} + (\alpha u u_x)_x - \beta u_{xxxx} = \gamma u$ can be written in the following multi-symplectic form~\cite{MYM11} by setting $z = (\phi \; u \; v \; w)$:
\begin{equation*}
    \begin{split}
        -\frac{1}{2} u_t - w_x &= -\gamma \phi, \\
       \frac{1}{2}\phi_t - v_x &= w -\frac{\alpha}{2} u^2, \\
       u_x &= \frac{v}{\beta} ,\\
       \phi_x &= u.
    \end{split}
\end{equation*}

Again, we assume small solutions and ignore the nonlinear term $u^2$. Then the error-propagation graph resembles Fig.~\ref{fig:graph_errors_Ostrovsky}.
\begin{figure}[htbp]
    \begin{center}
        \includegraphics{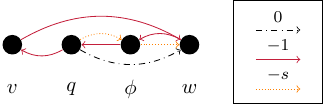}
    \end{center}
    \caption{The error propagation graph for the Ostrovsky equation $u_{tx}+ (\alpha u u_x)_x - \beta u_{xxxx}=\gamma u$}
    \label{fig:graph_errors_Ostrovsky}
\end{figure}
This graph has many negative cycles, and the diamond scheme is not appropriate.

At the end of this section, we note two important facts. First, the same analysis can be applied to high-order diamond schemes, observing the update formulae of diamond in the same manner as for the simple diamond scheme. We expect that the same stability condition on $\Delta t$ as for the simple diamond scheme will hold in this case as well, but a detailed study is left for future work.
Second, Theorem~\ref{thm:graph_stability} claims that, if the simple diamond scheme is to be stable for a PDE, then we need to choose $s$ (in the condition $\Delta t \sim (\Delta x)^s$) such that the error-propagation graph does not have any negative cycle for the choice of $s$.
Note that this is just a necessary condition for stability, and is generally not sufficient.
In the case of the wave equation, the condition $ s\ge 1$  is also sufficient, as is widely known.
On the other hand, as we will see later, for the ``good'' Boussinesq equation, the necessary condition is not sufficient.
In fact, the graph analysis yields a necessary condition $s \ge 2$, while numerical experiments show that actually we need $\Delta t = O(\Delta x^3)$, namely $s \ge 3$.
To fill the gap between the necessary and sufficient conditions, we need the precise eigenvalue analysis presented in the next section.
Still, the graph analysis has the strong advantage that it can be performed very quickly from the multi-symplectic form and filters out any unconditionally unstable cases (and also detects cases where very severe stability conditions, $\Delta t = O(\Delta x^s)$ with large $s$, are required).
The eigenvalue analysis in the next section provides more precise information on stability but requires constructing scheme matrices and computing their numerical eigenvalues.

\subsection{A shortcut recipe for Steps 1 and 2}

Once we understood the analyses in Steps 1 and 2, we see that they can be done simultaneously only by carefully watching the continuous system~\eqref{eq:multi-symplectic}.
The following gives a practical recipe.

\begin{enumerate}
    \item {\bf Write down the concrete form of the equations.} There are $d$ equations in~\eqref{eq:multi-symplectic}.
    \item {\bf Locate terms corresponding to the unknown discrete variables.}
    The unknown variables in the discrete scheme arise from the time-derivative terms or the terms on the right-hand side (RHS). Mark them first (the \boxed{\rm boxed} terms in the preceding discussions).
    \item {\bf Find obvious directions of the computations.} If the left-hand side (LHS) does not include a time derivative, then the LHS does not include an unknown variable, and the computation should be from the LHS to the RHS to find a new value in the RHS.
    Similarly, if the RHS is $0$, then the equation defines an update from the RHS to the LHS.
    Write the direction in the system, and identify the marked variable as ``updated.''
    \item {\bf Repeatedly find the rest of the directions.} By the ``updated'' variables, some other equations should become updatable. Update them, and repeat the procedure until all the variables are updated.
    If, during the procedure, we find multiple formulae that attempt to update the same variable, or, conversely, a variable for which no updating formula exists, the system is structurally inconsistent.
    \item {\bf Identify the error amplification factors.} Spatial derivatives (such as in $w_x$ and $u_x$) amplify the error in the variable by the factor $1/\Delta x$.
    Time derivatives (such as $v_t$ and $u_t$) amplify the error in the other variables by the factor $\Delta t$.
    Write them in the updates.
\end{enumerate}

Steps 3 and 4 in the recipe manually construct the Dulmage--Mendelsohn decomposition of $\gbi$ for the target PDE.
This is possible since in the multi-symplectic PDEs, the system is small and sparse.
For example, for the wave equation case, we find the following diagram after the procedure ends. The number \ctext{$i$} means the order of computation.
This is enough to reproduce the error-propagation graph $\gred$ above and is much faster.
\begin{align}
    \left\{
    \begin{aligned}
        -\boxed{v_t} + w_x &= 0, & \qquad & (\ctext{1}\; v\leftarrow w\; (\text{amplification } \Delta t/\Delta x\; (\text{index } s-1)))\\
        \boxed{u_t} &= \boxed{v}, & \qquad &(\ctext{3}\; u\leftarrow v\; (\text{amplification }\Delta t\; (\text{index } s)))\\
        -u_x &= -\boxed{w} & \qquad &(\ctext{2}\; u \rightarrow w\; (\text{amplification }1/\Delta x\; (\text{index } -1))).
    \end{aligned}
    \right.
\end{align}

\section{Step 3: Stability Analysis by Eigenvalues} \label{sec:eigen}

When we need a tight sufficient condition for stability, we may need precise eigenvalue analysis, as shown below.
It serves as a refinement of the graph-based necessary condition.
We assume that the target PDE is linear.

\subsection{Simple diamond scheme case}

To explicitly state the analysis procedure, let us consider the case where the multi-symplectic PDE~\eqref{eq:multi-symplectic} has three variables $z = (u \; v \; w)$.
The analysis easily extends to other cases.
Let $\bm{Z}^0$ be the vector of all the numerical solutions of $z$ at each mesh point in the initial state, namely
\[
\bm{Z}^0 = (u_0^0, v_0^0, w_0^0, u_0^{1/2}, v_0^{1/2}, w_0^{1/2}, u_1^0, v_1^0, w_1^0, \cdots, u_{N-1}^0, v_{N-1}^0, w_{N-1}^0, u_{N-1}^{1/2}, v_{N-1}^{1/2}, w_{N-1}^{1/2})^\top,
\]
where $u_i^0$ is the numerical solution of $u$ at the mesh point $(i \Delta x, 0)$, and $u_i^{1/2}$ is at $((i+1/2) \Delta x, \Delta t/2)$ (namely,  $\bm{Z}^0$ represents the cyan zig-zag line in Fig.~\ref{fig:diamond_mesh_zigzag1}). Note that we can skip the solution at $x = \ell$ ($i = N$) since we assume the periodic boundary condition.

\begin{figure}[htbp]
    \centering
    \includegraphics{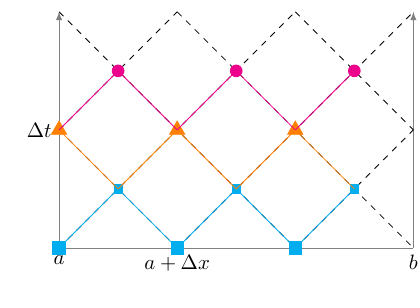}
    \caption{The diamond mesh with zig-zag line}
    \label{fig:diamond_mesh_zigzag1}
\end{figure}
The first half step of the simple diamond scheme updates the solutions on the cyan zig-zag line to those on the orange zig-zag line to find
\[
\bm{Z}^{1/2} = (u_0^1, v_0^1, w_0^1, u_0^{1/2}, v_0^{1/2}, w_0^{1/2}, u_1^1, v_1^1, w_1^1, \cdots, u_{N-1}^1, v_{N-1}^1, w_{N-1}^1, u_{N-1}^{1/2}, v_{N-1}^{1/2}, w_{N-1}^{1/2})^\top.
\]
Since now we assume that the target PDE is linear, the simple diamond scheme~\eqref{eq:diamond_update} is linear in terms of the unknown variables, and thus, by rearranging the terms, we find an update formula of the form
\[
\bm{Z}^{1/2} = M_1 \bm{Z}^0,
\]
where $M_1$ is a $6N \times 6N$ matrix expressing the diamond scheme for the target PDE.
To derive the structure of $M_1$ explicitly, we first write the single-diamond update in a linear-algebraic form. Let $z_i^0 := (u_i^0, v_i^0, w_i^0)^\top$, $z_i^{1/2} := (u_i^{1/2}, v_i^{1/2}, w_i^{1/2})^\top$, and $z_i^{1} := (u_i^{1}, v_i^{1}, w_i^{1})^\top$. Consider the diamond whose bottom, left, right, and top vertices are, respectively,
\[
z^b = z_i^0,\qquad z^l = z_{i-1}^{1/2},\qquad z^r = z_{i}^{1/2},\qquad z^t = z_i^{1},
\]
where the periodic boundary condition implies $z_{-1}^{1/2}=z_{N-1}^{1/2}$.
Assuming linearity $\nabla S(z)=Pz$ with a constant matrix $P$, and using $z^0=(z^t+z^b+z^l+z^r)/4$ in~\eqref{eq:diamond_update}, we obtain
\[
\left( \frac{1}{\Delta t} K - \frac{1}{4} P \right) z^t
=\left( \frac{1}{\Delta t} K + \frac{1}{4} P \right) z^b +
\left( \frac{1}{\Delta x} L + \frac{1}{4} P \right) z^l +
\left(- \frac{1}{\Delta x} L + \frac{1}{4} P \right) z^r .
\]
When the scheme is solvable, $\left( \frac{1}{\Delta t}K - \frac{1}{4}P \right)$ is invertible, and hence
\[
z^t = B z^b + A^- z^l + A^+ z^r,
\]
with the $3\times 3$ blocks defined by
\[
B = \left(\frac{1}{\Delta t}K-\frac{1}{4}P\right)^{-1}\left(\frac{1}{\Delta t}K+\frac{1}{4}P\right),\quad
A^- = \left(\frac{1}{\Delta t}K-\frac{1}{4}P\right)^{-1}\left(\frac{1}{\Delta x}L+\frac{1}{4}P\right),\quad
A^+ = \left(\frac{1}{\Delta t}K-\frac{1}{4}P\right)^{-1}\left(-\frac{1}{\Delta x}L+\frac{1}{4}P\right).
\]
For example, in the case of the wave equation, $A^-, B, A^+$ are given by
\[
B = \begin{pmatrix}
    1 & \frac{\Delta t}{2} & \\
    & 1 & \\
    & & -1
\end{pmatrix}, \qquad
A^- = \begin{pmatrix}
    & \frac{\Delta t}{4} & -\frac{\Delta t^2}{4\Delta x} \\
    & & -\frac{\Delta t}{\Delta x} \\
    -\frac{4}{\Delta x} & & -1
\end{pmatrix}, \qquad
A^+ = \begin{pmatrix}
    & \frac{\Delta t}{4} & \frac{\Delta t^2}{4\Delta x} \\
    & & \frac{\Delta t}{\Delta x} \\
    \frac{4}{\Delta x} & & -1
\end{pmatrix}
\]
Substituting $(z^t,z^b,z^l,z^r)=(z_i^{1},z_i^{0},z_{i-1}^{1/2},z_i^{1/2})$ gives the update on the orange zig-zag line:
\[
z_i^{1}=Bz_i^{0}+A^-z_{i-1}^{1/2}+A^+z_{i}^{1/2}\qquad (i=0,\dots,N-1).
\]
Moreover, in the first half step the half-integer values in $\bm{Z}_0$ are carried over to $\bm{Z}_{1/2}$ without change, namely
\[
z_i^{1/2} = I\, z_i^{1/2} \qquad (i=0,\dots,N-1),
\]
which generates the identity blocks in $M_1$. Stacking the above two equations for all $i$ in the ordering of $\bm{Z}^0$ and $\bm{Z}^{1/2}$ yields a block-circulant $M_1$. Note that the periodicity makes the term $A^-z_{i-1}^{1/2}$ appear in the upper-right corner when $i=0$. Consequently, $M_1$ takes the block diagonal and circulant form
\[
M_1 =
\begin{pmatrix}
    B & A^+ &  & & & & & & A^-\\
    & I & \\
    & A^- & B & A^+ & \\
    &     &   & I & \\
    &     &   &   & & \\
    &     &   &   & & \ddots \\
    &     &   &   & & \\
    & &     &   &   & & A^- & B & A^+ \\
    & &     &   &   & &  &  & I
\end{pmatrix}
\]
where $I$ is the identity matrix of size $3 \times 3$, and the blank elements are the zero matrices $O$ of the corresponding size.

Next, the solutions on the magenta zig-zag line in Fig.~\ref{fig:diamond_mesh_zigzag1}:
\[
\bm{Z}^1 =  (u_0^1, v_0^1, w_0^1, u_0^{3/2}, v_0^{3/2}, w_0^{3/2}, u_1^1, v_1^1, w_1^1, \cdots, u_{N-1}^1, v_{N-1}^1, w_{N-1}^1, u_{N-1}^{3/2}, v_{N-1}^{3/2}, w_{N-1}^{3/2})^\top
\]
should be determined from $\bm{Z}^{1/2}$ (those on the orange zig-zag line in Fig.~\ref{fig:diamond_mesh_zigzag1}) as
\[
\bm{Z}^{1} = M_2 \bm{Z}^{1/2},
\]
where $M_2$ is a $6N \times 6N$ matrix, again expressing the diamond scheme. Similarly to $M_1$,  $M_2$ takes the form
\[
M_2 =
\begin{pmatrix}
    I &  &  & & & & & \\
    A^- & B & A^+ \\
    & & I \\
    & & A^-& B & A^+ \\
    &     &   &   & & \\
    &     &   &   & & \ddots \\
    &     &   &   & & \\
    &     &   &   & &  & & I &  \\
    A^+ &     &   &   & &  &  & A^- & B
\end{pmatrix}.
\]

Concatenating the updating matrices $M_1$  and $M_2$, we find the updating matrix
\[
M = M_2 M_1,
\]
which advances the numerical solutions by $\Delta t$.
The general numerical solutions at time $t = j \Delta t$ are expressed as $\bm{Z}^j = M^j \bm{Z}^0$.
The eigenvalues of $M$ identify the growth rate of the numerical errors in the numerical solutions.

The eigenvalue problem can be decomposed into smaller problems if we take the circulant structure of $M_1$ and $M_2$, and accordingly $M$, into account.
Let $X_1, Y_1, X_2, Y_2$ be $6 \times 6$ matrices defined as follows.
\[
X_1 = \begin{pmatrix}
    B & A^+ \\
    O & I
\end{pmatrix},\;
Y_1 = \begin{pmatrix}
    O & A^- \\
    O &  O
\end{pmatrix},\;
X_2 = \begin{pmatrix}
    I & O  \\
    A^- & B
\end{pmatrix},\;
Y_2 = \begin{pmatrix}
    O & O \\
    A^+ & O
\end{pmatrix}.
\]
Then, $M_1$ and $M_2$ can be rewritten as
\[
M_1 = \begin{pmatrix}
    X_1 &     & & & & & Y_1  \\
    Y_1 & X_1 & \\
    & Y_1 & X_1 & \\
    &     & & & \\
    &     & & & \ddots \\
    &     & & & & \\
    &     & & & & Y_1 & X_1\\
\end{pmatrix}, \qquad
M_2 = \begin{pmatrix}
    X_2 & Y_2    & & & & &   \\
    & X_2 & Y_2 \\
    & & X_2 & Y_2 & \\
    &     & & & \\
    &     & & & \ddots \\
    &     & & & & \\
    Y_2&     & & & & & X_2\\
\end{pmatrix}.
\]
Accordingly, $M = M_2 M_1$ becomes
\begin{equation}\label{eq:update_eq_simple_diamond}
M =
\begin{pmatrix}
    X_2 X_1 + Y_2 Y_1 & Y_2 X_1 & & & & X_2 Y_1 \\
    X_2 Y_1 & X_2 X_1 + Y_2 Y_1 & Y_2 X_1 & \\
    & X_2 Y_1 & X_2 X_1 + Y_2 Y_1 & Y_2 X_1 & \\
    \\
    & & & & \ddots \\
    \\
    Y_2 X_1& & & & X_2Y_1 & X_2 X_1 + Y_2 Y_1
\end{pmatrix}.
\end{equation}
Denoting $C_0 = X_2 X_1 + Y_2 Y_1$, $C^+ = Y_2 X_1$, and $C^- = X_2 Y_1$, we can further rewrite $M$ as
\[
M=
\begin{pmatrix}
    C_0 & C^+ & & & & C^- \\
    C^- & C_0 & C^+ & \\
    & C^- & C_0 & C^+ & \\
    \\
    & & & & \ddots \\
    \\
    C^+& & & & C^- & C_0
\end{pmatrix}.
\]
This block circulant matrix is similar to the following block diagonal matrix:
\[
\Lambda = \mathrm{diag} \; (\Lambda_0, \; \Lambda_1, \ldots , \Lambda_{N-1}),
\]
where
\[
\Lambda_k = C_0 + \zeta_N^k C^+ + \zeta_N^{-k} C^-
\]
and $\zeta_N$ is the $N$-th root of unity, i.e. $\zeta_N = \exp(2 \pi i / N)$.
Thus, it is sufficient to numerically calculate the eigenvalues of  $\Lambda_k$ (k=0,1,\ldots, N-1), which are $6 \times 6$ matrices.
Hereinafter, let $\lambda_1$ denote the dominant eigenvalue of $M$, that is, the eigenvalue that has the largest absolute value. The scheme is stable if and only if $|\lambda_1| \leq 1$.
This can give a tighter stability estimate (see the examples below).

\subsection{Higher-order diamond scheme case} \label{subsec:high-order}
The analysis is similar to the previous case, but the structure of the update matrix becomes more complicated.

Let us first identify the updating matrix in each diamond.
In each diamond, one inputs $\{\tilde{z^\mathrm{b}}, \tilde{z_\mathrm{l}}\}$, and outputs $\{\tilde{z^\mathrm{t}}, \tilde{z_\mathrm{r}}\}$.
There, we first compute the collocated values $Z_i^j$ by solving the following equation:
\begin{equation}\label{eq:update_diamond_RK}
    \nabla S(\bm{Z}_i^j) = \widetilde{K} f_{jk} \bm{Z}_i^k + \widetilde{L} f_{ik} \bm{Z}_k^j - \widetilde{K} \left(\sum_k f_{jk}\right) \tilde{z}^\mathrm{b}_i - \widetilde{L} \left(\sum_k f_{ik}\right) \tilde{z}_\mathrm{l}^j.
\end{equation}
Due to the linearity, we can write $\nabla S(\bm{Z}_i^j)=P\bm{Z}_i^j$ with some constant matrix $P$.
Let us also introduce the following notation.
\begin{equation*}
    \begin{split}
        \bm{\mu} &= \left(\sum_{k = 1}^r f_{1k}, \sum_{k = 1}^r f_{2k}, \ldots , \sum_{k = 1}^r f_{rk}  \right)^\top, \\
        \bm{\beta} &= \bm{b}^\top F, \\
        Q &= I_{r^2} \otimes P - I_r \otimes F \otimes \widetilde{K} - F \otimes I_r \otimes \widetilde{L}, \\
        D_\mathrm{b} &= -(I_r \otimes \bm{\mu} \otimes I_d) (I_r \otimes \widetilde{K}), \\
        D_\mathrm{l} &= -(\bm{\mu} \otimes I_r \otimes I_d) (I_r \otimes \widetilde{L}), \\
    \end{split}
\end{equation*}
where $F = (f_{ij})_{i,j=1}^r$ is the inverse of the Runge--Kutta matrix.
Then Eq.~\eqref{eq:update_diamond_RK} can be rewritten in the following form:
\begin{align} \label{eq:high-order:system}
Q\bm{z} = D_\mathrm{b} \tilde{z}^\mathrm{b} + D_\mathrm{l} \tilde{z}_\mathrm{l},
\end{align}
where $\bm{z} = (Z_1^1, Z_1^2, \ldots , Z_1^r, Z_2^1, \ldots , Z_2^r, \ldots , Z_r^1, \ldots , Z_r^r)$.
After solving this equation for $z$, the outputs $\tilde{z}^\mathrm{t}, \tilde{z}_\mathrm{r}$ can be computed as
\begin{equation*}
        \tilde{z}_\mathrm{r}^i = \tilde{z}_\mathrm{l}^i + \sum_{k = 1}^r b_k \sum_{p = 1}^r f_{kp} \left( Z_p^i - \tilde{z}_\mathrm{l}^i  \right), \qquad
        \tilde{z}^\mathrm{t}_i = \tilde{z}^\mathrm{b}_i + \sum_{k = 1}^r b_k \sum_{p = 1}^r f_{kp} \left( Z_i^p - \tilde{z}^\mathrm{b}_i  \right).
\end{equation*}
The first formula can be simplified as
\[
\tilde{z}_\mathrm{r}^i = \tilde{z}_\mathrm{l}^i + \sum_{k = 1}^r b_k \left( -\mu_k \tilde{z}_\mathrm{l}^i + \sum_{p = 1}^r f_{kp} Z_p^i \right) = (1 - \alpha ) \tilde{z}_\mathrm{l}^i + \sum_{k = 1}^r b_k \sum_{p = 1}^r f_{kp} Z_p^i,
\]
where $\alpha = \bm{b}^\top \bm{\mu} \; (= 1 - (-1)^r)$.
Hence we see
\[
\tilde{z}_\mathrm{r}=(1 - \alpha)\tilde{z}_\mathrm{l} + (\bm{\beta} \otimes I_r \otimes I_d) \bm{z}.
\]
Similarly, we see
\[
\tilde{z}^\mathrm{t}=(1 - \alpha)\tilde{z}^\mathrm{b} + (I_r \otimes \bm{\beta} \otimes I_d) \bm{z}.
\]
Summing up, we find the updating formula in one diamond as
\[
\begin{pmatrix}
    \tilde{z}^\mathrm{t} \\
    \tilde{z}_\mathrm{r}
\end{pmatrix}
=
\begin{pmatrix}
    (I_r \otimes \bm{\beta} \otimes I_d) Q^{-1} D_\mathrm{l} & (1-\alpha) I_{dr} + (I_r \otimes \bm{\beta} \otimes I_d) Q^{-1} D_\mathrm{b} \\
    (1-\alpha) I_{dr} + (\bm{\beta} \otimes I_{dr}) Q^{-1} D_\mathrm{l} & (\bm{\beta} \otimes I_{dr}) Q^{-1} D_\mathrm{b}
\end{pmatrix}
\begin{pmatrix}
    \tilde{z}_\mathrm{l} \\
    \tilde{z}^\mathrm{b}
\end{pmatrix}
=:
\begin{pmatrix}
    C_\mathrm{l}^\mathrm{t} & C_\mathrm{b}^\mathrm{t} \\
    C_\mathrm{l}^\mathrm{r} & C_\mathrm{b}^\mathrm{r}
\end{pmatrix}
\begin{pmatrix}
    \tilde{z}_\mathrm{l} \\
    \tilde{z}^\mathrm{b}
\end{pmatrix}.
\]
The last equality defines the matrices $C$'s.

Now we can write the first-half update matrix $\tilde{M}_1$, which is the matrix that updates the cyan zig-zag line in Fig.~\ref{fig:diamond_mesh_zigzag1} to the orange zig-zag line, and the second-half matrix $\tilde{M}_2$, which updates the orange zig-zag line to the magenta zig-zag line, as
\[
\tilde{M}_1 = \begin{pmatrix}
    C_\mathrm{b}^\mathrm{r} & & & & C_\mathrm{l}^\mathrm{r} \\
    & C_\mathrm{l}^\mathrm{t} & C_\mathrm{b}^\mathrm{t} & &  \\
    & C_\mathrm{l}^\mathrm{r} & C_\mathrm{b}^\mathrm{r} & &  \\
    & & & \ddots & \\
    C_\mathrm{b}^\mathrm{t} & & & & C_\mathrm{l}^\mathrm{t} \\
\end{pmatrix}, \qquad
\tilde{M}_2 = \begin{pmatrix}
    C_\mathrm{l}^\mathrm{t} & C_\mathrm{b}^\mathrm{t} &  & &  \\
    C_\mathrm{l}^\mathrm{r} & C_\mathrm{b}^\mathrm{r} &  & &  \\
    &  & \ddots & \\
    &  & & \ddots & \\
    & & & & C_\mathrm{l}^\mathrm{t} & C_\mathrm{b}^\mathrm{t} \\
    & & & & C_\mathrm{l}^\mathrm{r} & C_\mathrm{b}^\mathrm{r} \\
\end{pmatrix},
\]
respectively. Therefore, the overall update matrix of the high-order diamond scheme is given by
\begin{equation}\label{eq:update_eq_high-order_diamond}
\tilde{M} := \tilde{M}_2 \tilde{M}_1 = \begin{pmatrix}
    C_\mathrm{l}^\mathrm{t} C_\mathrm{b}^\mathrm{r} & C_\mathrm{b}^\mathrm{t} C_\mathrm{l}^\mathrm{t} & {C_\mathrm{b}^\mathrm{t}}^2 & & & C_\mathrm{l}^\mathrm{t} C_\mathrm{l}^\mathrm{r}  \\
    C_\mathrm{l}^\mathrm{r} C_\mathrm{b}^\mathrm{r} & C_\mathrm{b}^\mathrm{r} C_\mathrm{l}^\mathrm{t} & C_\mathrm{b}^\mathrm{r} C_\mathrm{b}^\mathrm{t} & & & {C_\mathrm{l}^\mathrm{r}}^2 \\
    & & \ddots & \\
    & & & \ddots & \\
    {C_\mathrm{b}^\mathrm{t}}^2 & & & C_\mathrm{l}^\mathrm{t} C_\mathrm{l}^\mathrm{r} & C_\mathrm{l}^\mathrm{t} C_\mathrm{b}^\mathrm{r} & C_\mathrm{b}^\mathrm{t} C_\mathrm{l}^\mathrm{t} \\
    C_\mathrm{b}^\mathrm{r} C_\mathrm{b}^\mathrm{t} & & & {C_\mathrm{l}^\mathrm{r}}^2 & C_\mathrm{l}^\mathrm{r} C_\mathrm{b}^\mathrm{r} & C_\mathrm{b}^\mathrm{r} C_\mathrm{l}^\mathrm{t} \\
\end{pmatrix}.
\end{equation}
Denoting
\[
\tilde{C}_0 = \begin{pmatrix}
    C_\mathrm{l}^\mathrm{t} C_\mathrm{b}^\mathrm{r} & C_\mathrm{b}^\mathrm{t} C_\mathrm{l}^\mathrm{t} \\
    C_\mathrm{l}^\mathrm{r} C_\mathrm{b}^\mathrm{r} & C_\mathrm{b}^\mathrm{r} C_\mathrm{l}^\mathrm{t}
\end{pmatrix},
\quad
\tilde{C}^+ = \begin{pmatrix}
    {C_\mathrm{b}^\mathrm{t}}^2 & O \\
    C_\mathrm{b}^\mathrm{r} C_\mathrm{b}^\mathrm{t} & O
\end{pmatrix},
\quad
\tilde{C}^- = \begin{pmatrix}
    O & C_\mathrm{l}^\mathrm{t} C_\mathrm{l}^\mathrm{r} \\
    O & {C_\mathrm{l}^\mathrm{r}}^2
\end{pmatrix},
\]
$\tilde{M}$ can be written as the block circulant matrix where the block diagonal component is $\tilde{C}_0$ and the block sub-diagonal components are $\tilde{C}^+$ and $\tilde{C}^-$. Therefore, $\tilde{M}$ is similar to the block diagonal matrix:
\[
\tilde{\Lambda} = \mathrm{diag} \; (\tilde{\Lambda}_0, \; \tilde{\Lambda}_1, \ldots , \tilde{\Lambda}_{N-1}),
\]
where
\[
\tilde{\Lambda}_k = \tilde{C}_0 + \zeta_N^k \tilde{C}^+ + \zeta_N^{-k} \tilde{C}^-.
\]
The eigenvalues of $\tilde{M}$ can be calculated from $\tilde{\Lambda}_k$, which is a $2dr \times 2dr$ matrix. As in the case of the simple diamond scheme, let $\tilde{\lambda}_1$ denote the dominant eigenvalue of $\tilde{M}$. The scheme is stable if and only if $|\tilde{\lambda}_1| \leq 1$.

\section{Numerical experiments for various multi-symplectic PDEs} \label{sec:numerical}
In this section, we present the stability analysis and numerical experiments for the diamond schemes applied to various multi-symplectic PDEs.

\subsection{PDEs that are stable under $\Delta t = O (\Delta x)$}
For some multi-symplectic PDEs, including the Klein--Gordon equations originally considered in McLachlan and Wilkins~\cite{Diamond}, the diamond schemes are stable under $\Delta t = O(\Delta x)$.

\subsubsection{The Klein--Gordon equation}
McLachlan and Wilkins~\cite{Diamond} presented a dispersion stability analysis for the wave equation.
No stability analysis has been given for more general Klein--Gordon type equations.
Let us consider the linear Klein--Gordon equation $u_{tt} - u_{xx} = u$.

\begin{figure}[htbp]
    \begin{center}
        \includegraphics{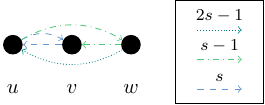}
    \end{center}
    \caption{Error propagation graph for the linear Klein--Gordon equation}
    \label{fig:stability_graph_LKG}
\end{figure}

The linear Klein--Gordon equation can be written in a multi-symplectic form~\eqref{eq:multi-symplectic} with~\eqref{eq:LKG} and $V(u)=u^2/2$.
Fig.~\ref{fig:stability_graph_LKG} is its error propagation graph, which shows the necessary condition $s \ge 2/3$.
More detailed eigenvalue analysis reveals the sufficient condition $s \ge 1$,
i.e.,~$\Delta t = O(\Delta x)$, which can be shown as follows.
The dominant eigenvalues corresponding to the block diagonal component $\Lambda_0$ have absolute values $1 + O(\Delta t)$. However, $\Lambda_0$ corresponds to frequency $0$, that is, the situation where all diamonds have exactly the same error. In practice, this situation never holds, so the dominant eigenvalue in practice corresponds $\Lambda_k$ where $1 \le k \le N-1$, that is equal to $1$ when $\Delta t < \Delta x$. This sufficient condition is also valid for the high-order diamond schemes.

\subsubsection{The nonlinear Dirac equation}
The nonlinear Dirac equation of the following form:
\begin{equation}
    \begin{split}
        \pdv{\psi_1}{t} + \pdv{\psi_2}{x} + \mathrm{i}m\psi_1 + 2\mathrm{i}\lambda (|\psi_2|^2 - |\psi_1|^2) \psi_1 &= 0, \\
        \pdv{\psi_2}{t} + \pdv{\psi_1}{x} - \mathrm{i}m\psi_2 + 2\mathrm{i}\lambda (|\psi_1|^2 - |\psi_2|^2) \psi_2 &= 0,
    \end{split}
\end{equation}
where $\psi_1$ and $\psi_2$ are complex-valued functions and $m$ and $\lambda$ are real constants, can be written in the multi-symplectic form~\cite{HL06} by setting $z = (p_1, q_1, p_2, q_2) = (\Re \psi_1, \Im \psi_1, \Re \psi_2, \Im \psi_2)$,
\begin{equation}
    \begin{split}
        {p_1}_t + {p_2}_x &= \phantom{-{}} mq_1 + 2\lambda (p_2^2 + q_2^2 - p_1^2 - q_1^2)q_1, \\
        {q_1}_t + {q_2}_x &= -mp_1 - 2\lambda (p_2^2 + q_2^2 - p_1^2 - q_1^2)p_1, \\
        {p_2}_t + {p_1}_x &= -mq_2 - 2\lambda (p_2^2 + q_2^2 - p_1^2 - q_1^2)q_2, \\
        {q_2}_t + {q_1}_x &= \phantom{-{}}m p_2 + 2\lambda (p_2^2 + q_2^2 - p_1^2 - q_1^2)p_2 .
    \end{split}
\end{equation}

Fig.~\ref{fig:stability_graph_NLD} shows the error propagation graph for the linearized Dirac equation. The necessary condition for the stability of the diamond schemes is $s > 1/2$. As in the linear Klein--Gordon case, the eigenvalue stability analysis shows a sufficient condition $\Delta t = O(\Delta x)$.
In fact, we numerically verified that the absolute values of all eigenvalues of the update matrix $M$ for the simple diamond scheme do not exceed $1$ whenever $\Delta t < \Delta x$.
We also observed the same for the higher-order diamond scheme.
The condition $\Delta t = O(\Delta x)$ is mild enough, and the diamond scheme is practical.

\begin{figure}[htbp]
    \begin{center}
        \includegraphics{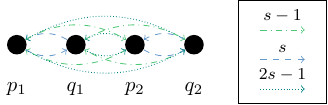}
    \end{center}
    \caption{The error propagation graph for the linear Dirac equation}
    \label{fig:stability_graph_NLD}
\end{figure}

Fig.~\ref{fig:experiment_NLD1_sol1} shows a numerical solution of $p_1$ obtained by the simple diamond scheme, where the initial solution was set to a breather soliton~\cite{HL06} in spatial range $ [-24, 24]$.
Other parameters are set to $\Delta x = 0.3, \; \Delta t = 0.2, \; T = 50$. The numerical solution is stable for a long time span. In addition, as shown in Fig.~\ref{fig:experiment_NLD1_error}, the conservation of the total energy holds (the energy is computed by the rule $\int_a^b E(x) \dd{x} \simeq \sum_{i = 1}^N E(x_i) \Delta x$). Setting $\Delta x = 0.075$ and $\Delta t = 0.05$ also leads to the same stable result, as shown in Fig.~\ref{fig:experiment_NLD2_sol} and Fig.~\ref{fig:experiment_NLD2_error}.

\begin{figure}[htbp]
    \begin{minipage}[t]{.48\linewidth}
        \begin{center}
            \includegraphics[width=75mm]{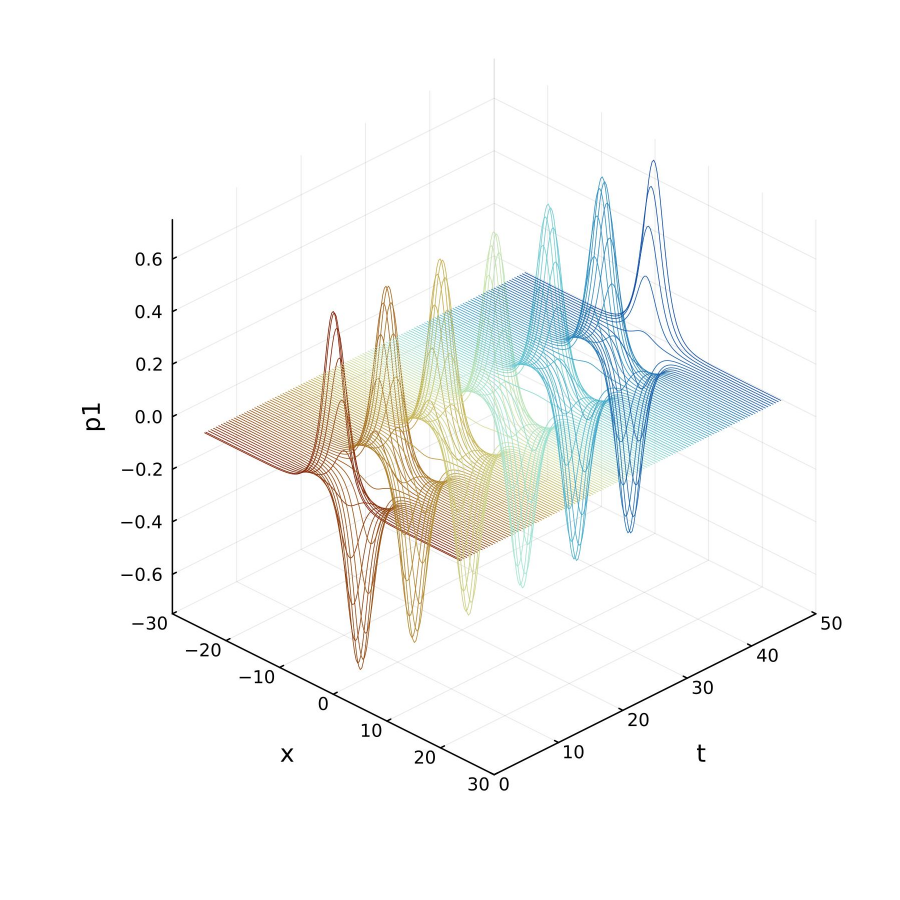}
            \caption{Numerical solutions for the Dirac equation obtained by the simple diamond scheme ($\Delta x = 0.3, \Delta t = 0.2$)}
            \label{fig:experiment_NLD1_sol1}
        \end{center}
    \end{minipage}
    \hfill
    \begin{minipage}[t]{.48\linewidth}
        \begin{center}
            \includegraphics[width=75mm]{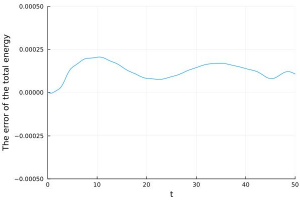}
            \caption{Time evolution of error of the total energy for the Dirac equation ($\Delta x \sim 0.3, \Delta t = 0.2$)}
            \label{fig:experiment_NLD1_error}
        \end{center}
    \end{minipage}
\end{figure}

\begin{figure}[htbp]
    \begin{minipage}[t]{.48\linewidth}
        \begin{center}
            \includegraphics[width=75mm]{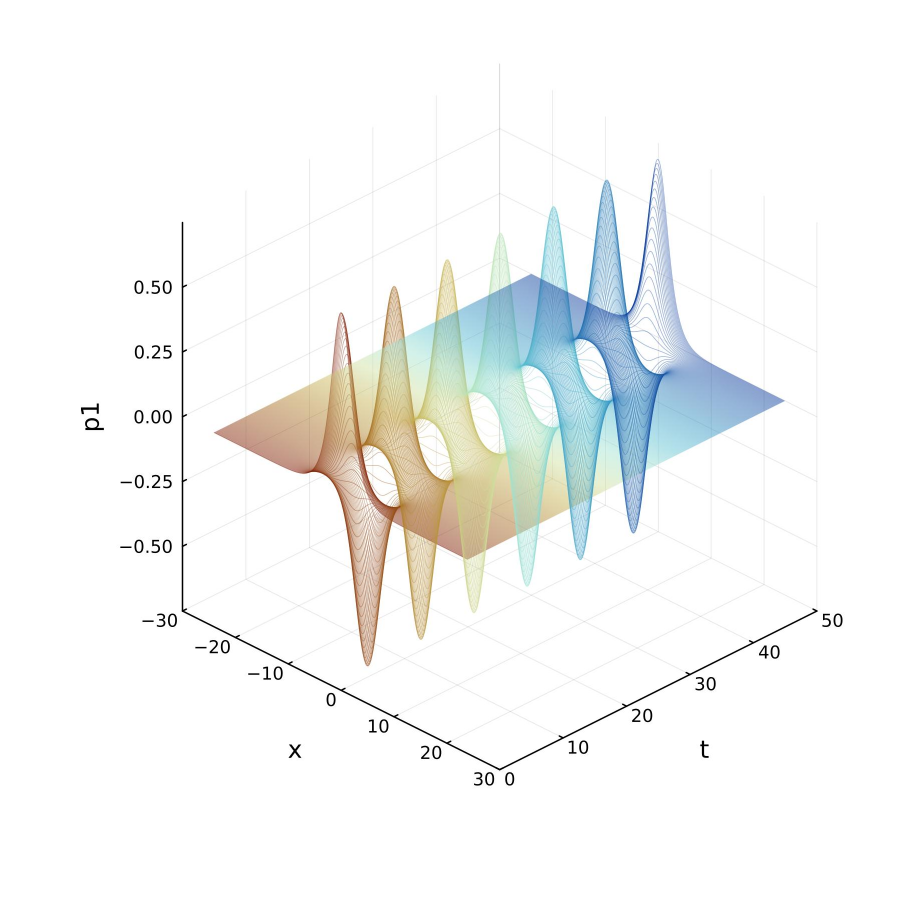}
            \caption{Numerical solutions for the Dirac equation obtained by the simple diamond scheme ($\Delta x = 0.075, \Delta t = 0.05$)}
            \label{fig:experiment_NLD2_sol}
        \end{center}
    \end{minipage}
    \hfill
    \begin{minipage}[t]{.48\linewidth}
        \begin{center}
            \includegraphics[width=75mm]{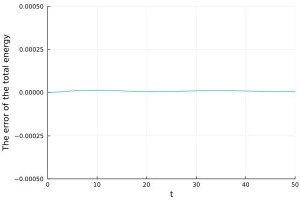}
            \caption{Time evolution of error of the total energy for the Dirac equation ($\Delta x = 0.075, \Delta t = 0.05$)}
            \label{fig:experiment_NLD2_error}
        \end{center}
    \end{minipage}
\end{figure}

\begin{figure}[htbp]
    \begin{minipage}[t]{.48\linewidth}
        \begin{center}
            \includegraphics[width=75mm]{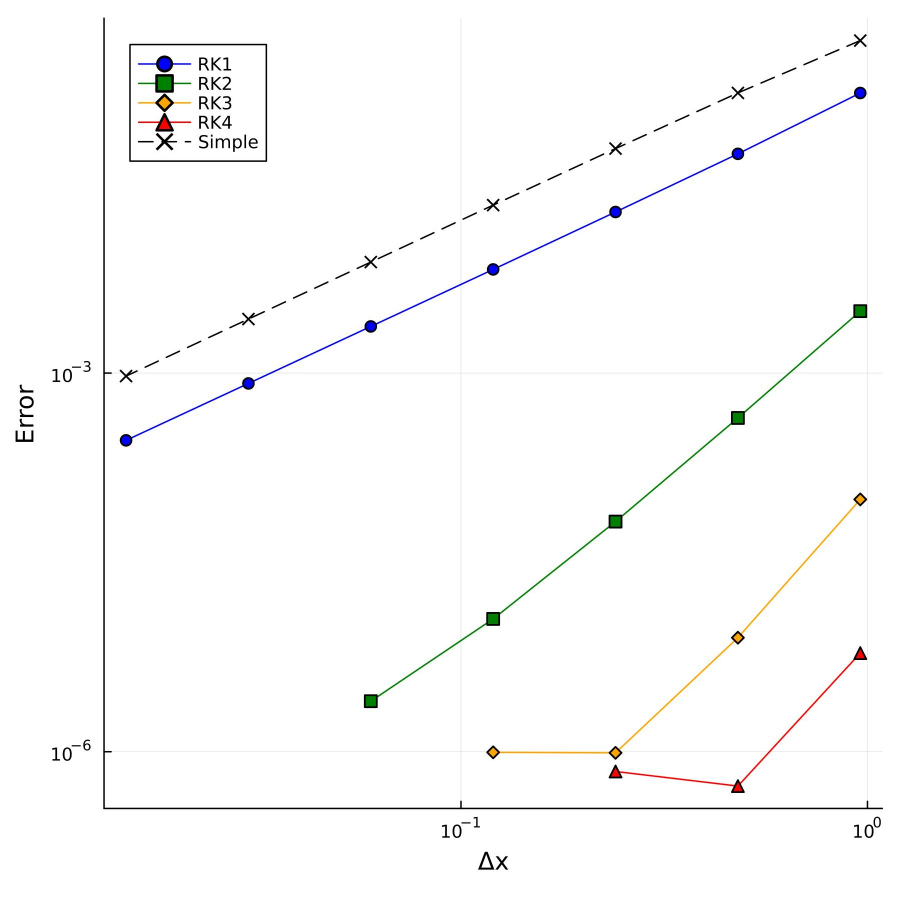}
            \caption{Global errors of numerical solutions for the Dirac equation obtained by the simple and high-order diamond schemes ($\kappa = \Delta t / \Delta x$ was set to $2/3$)}
            \label{fig:experiment_NLD_RK_error}
        \end{center}
    \end{minipage}
    \hfill
    \begin{minipage}[t]{.48\linewidth}
        \begin{center}
            \includegraphics[width=75mm]{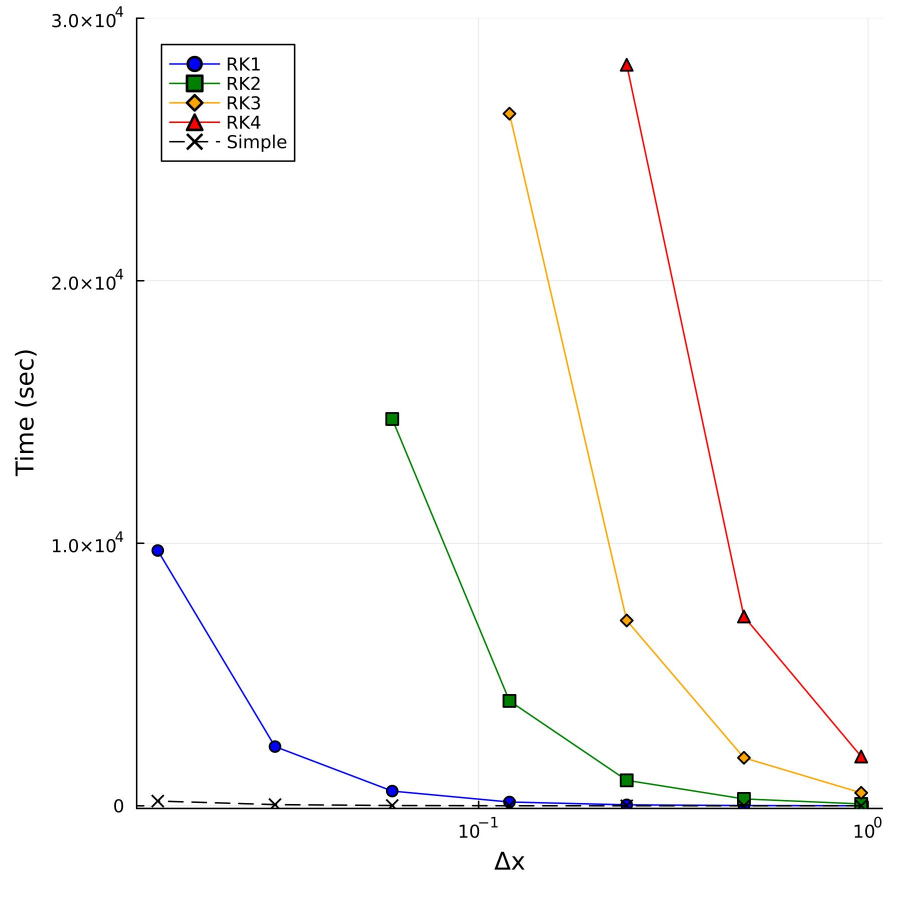}
            \caption{Computational times of the simple and high-order diamond schemes for the Dirac equation ($\kappa = \Delta t / \Delta x$ was set to $2/3$)}
            \label{fig:experiment_NLD_RK_time}
        \end{center}
    \end{minipage}
\end{figure}

\newpage

\subsection{PDEs that are stable under $\Delta t = O((\Delta x)^3)$}

\subsubsection{The ``good'' Boussinesq equation}
The ``good'' Boussinesq equation (or simply the Boussinesq equation) $u_{tt} - u_{xx} = -u_{xxxx} + (u^2)_{xx}$ can be written in the multi-symplectic form~\cite{Chen05} by setting $z = (u \; v \; p \; q)$,
\begin{equation*}
    \begin{split}
        -v_t - p_x &= -u - 2u^2, \\
        u_t - q_x &= 0,\\
        u_x &= p,\\
        v_x &= q.
    \end{split}
\end{equation*}

\begin{figure}[htbp]
    \begin{center}
        \includegraphics{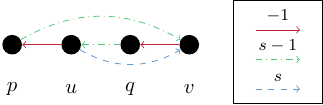}
    \end{center}
    \caption{The error propagation graph for the linearized ``good'' Boussinesq equation}
    \label{fig:stability_graph_Boussinesq}
\end{figure}

Fig.~\ref{fig:stability_graph_Boussinesq} illustrates the error propagation graph for the Boussinesq equation. From the graph analysis, the necessary condition for stability is $\Delta t = O(\Delta x^2)$, since the least-weight cycle in the graph has weight $2s - 4$. However, this condition is not sufficient. In fact, the eigenvalue analysis shows that $\Delta t = O(\Delta x^3)$ is sufficient for ensuring the stability of the diamond schemes as follows.
We neglect the nonlinear term $u^2$ and compute the maximum eigenvalue $\lambda_1$ of the update matrix $M$ for the diamond schemes numerically.
Fig.~\ref{fig:stable_dt_vs_dx_Boussinesq_simple} presents the relationship between $\Delta t$ and $\Delta x$ such that $|\lambda_1| \le 1$ in the case of the simple diamond scheme. Fig.~\ref{fig:stable_dt_vs_dx_Boussinesq_RK} shows the result for the higher-order diamond schemes.
From these numerical simulations, we observe the following: (1) we need $\Delta t \le c \Delta x^3$; (2) the constant $c$ depends almost linearly on the domain length $b - a$ and on $r^2$, where $r$ denotes the number of stages of the Runge--Kutta method.
When $\Delta x$ is small or $b - a$ is large, satisfying this stability condition becomes infeasible without employing multi-precision computations, rendering the scheme impractical.

Fig.~\ref{fig:experiment_Boussinesq1_sol} shows a numerical solution obtained by the simple diamond scheme on the spatial domain $[-50, 50]$ and the parameters $\Delta x = 0.1, \; \Delta t = 10^{-6}$, truncated into a range $[-5, 5]$. The initial condition is a one-soliton. This setting satisfies the numerically derived stability condition, and the multi-symplectic conservation of the total energy holds (Fig.~\ref{fig:experiment_Boussinesq1_error}).

\begin{figure}[htbp]
    \begin{minipage}[t]{.48\linewidth}
        \begin{center}
            \includegraphics[width=75mm]{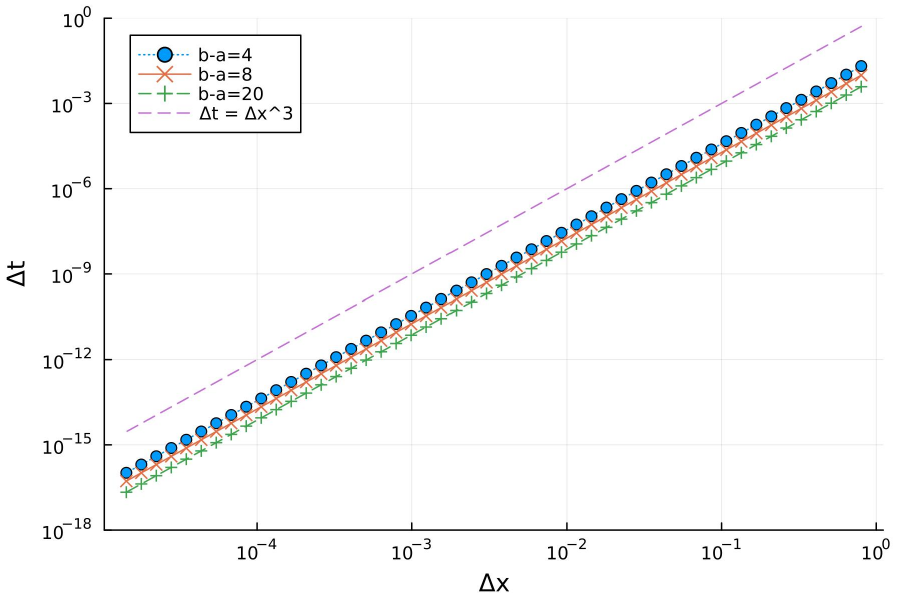}
            \caption{$\Delta t$ vs $\Delta x$ for the stability of the simple diamond scheme of the Boussinesq equation, $b - a$ denotes the length of the spatial range.}
            \label{fig:stable_dt_vs_dx_Boussinesq_simple}
        \end{center}
    \end{minipage}
    \hfill
    \begin{minipage}[t]{.48\linewidth}
        \begin{center}
            \includegraphics[width=75mm]{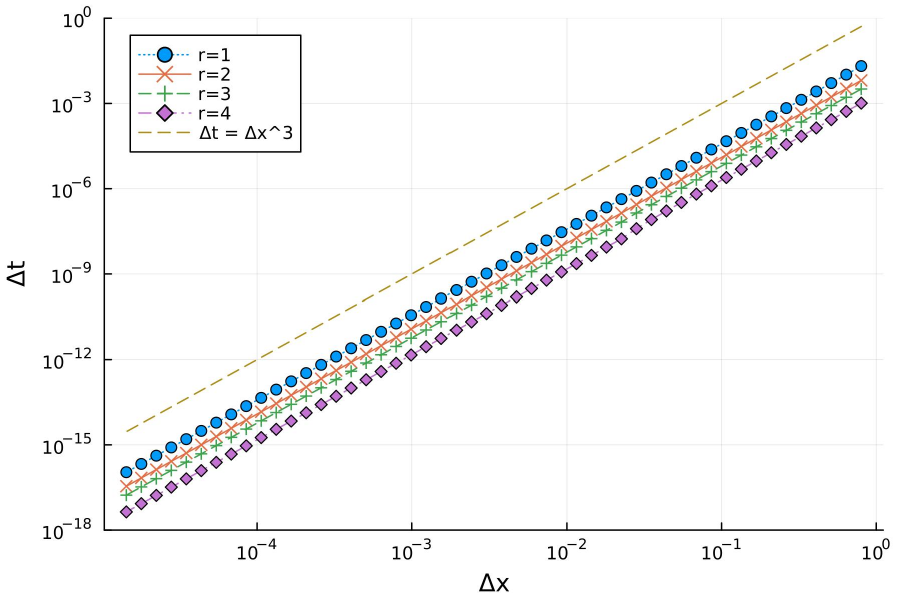}
            \caption{$\Delta t$ vs $\Delta x$ for the stability of the high-order diamond schemes of the Boussinesq equation, $r$ denotes the stage of the Runge--Kutta method and $b - a$ is fixed to $4$.}
            \label{fig:stable_dt_vs_dx_Boussinesq_RK}
        \end{center}
    \end{minipage}
\end{figure}

\begin{figure}[htbp]
    \begin{minipage}[t]{.48\linewidth}
        \begin{center}
            \includegraphics[width=75mm]{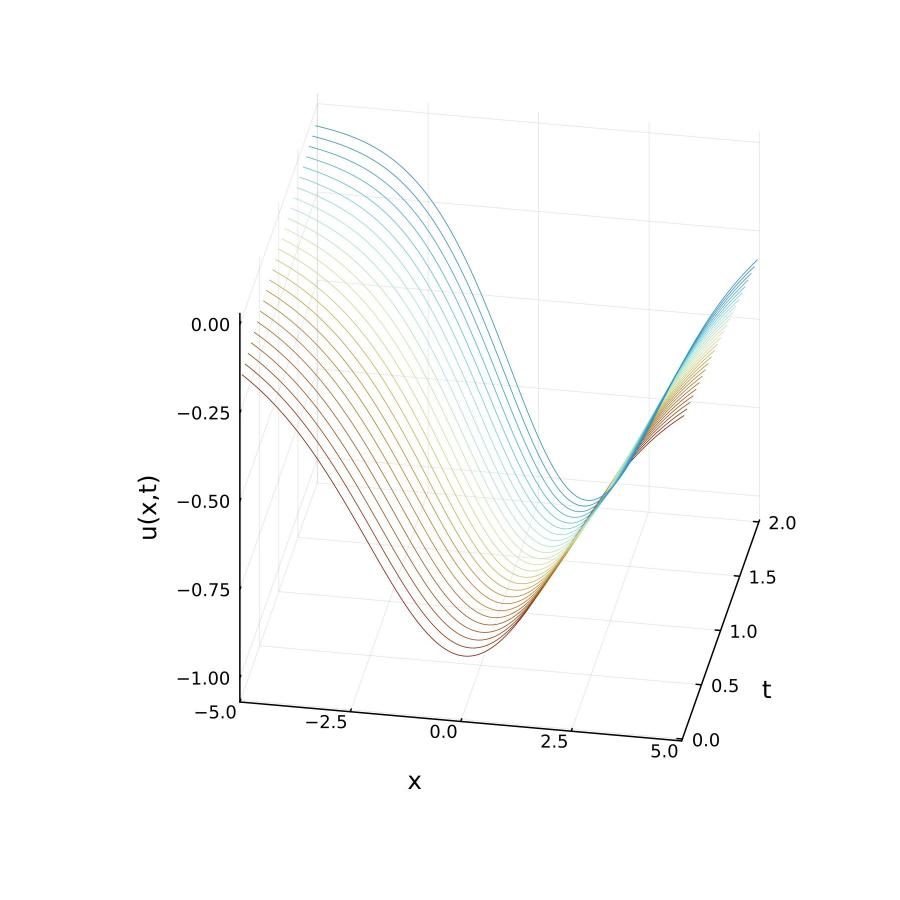}
            \caption{Numerical solutions for the Boussinesq equation obtained by the simple diamond scheme ($\Delta x = 0.1, \Delta t = 10^{-6}$), truncated into a range $[-5, 5]$}
            \label{fig:experiment_Boussinesq1_sol}
        \end{center}
    \end{minipage}
    \hfill
    \begin{minipage}[t]{.48\linewidth}
        \begin{center}
            \includegraphics[width=75mm]{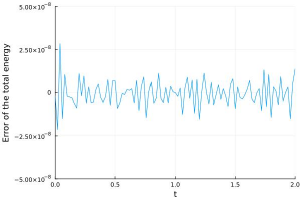}
            \caption{Time evolution of error of the total energy for the Boussinesq equation ($\Delta x = 0.1, \Delta t = 10^{-6}$)}
            \label{fig:experiment_Boussinesq1_error}
        \end{center}
    \end{minipage}
\end{figure}

\subsubsection{The nonlinear Schr\"{o}dinger equation}
The nonlinear Schr\"{o}dinger equation $i \phi_t + \phi_{xx} + a|\phi|^2 \phi = 0$ where $\phi = p + iq$ is complex-valued function can be written in the multi-symplectic form~\cite{CZT02} by setting $z = (p \; q \; v \; w)$,
\begin{equation*}
    \begin{split}
        q_t - v_x &= a p (p^2 + q^2), \\
       -p_t - w_x &= a q (p^2 + q^2), \\
       p_x &= v, \\
       q_x &= w.
    \end{split}
\end{equation*}

Linearizing the equations by assuming $p^2 + q^2 = \mathrm{const.}$, the graph stability analysis and the eigenvalue analysis can be carried out. Fig.~\ref{fig:stability_graph_Sch} shows the error propagation graph of the linear Schr\"{o}dinger equation.

\begin{figure}[htbp]
    \begin{center}
        \includegraphics{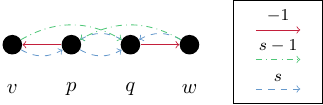}
    \end{center}
    \caption{The error propagation graph for the linear Schr\"{o}dinger equation}
    \label{fig:stability_graph_Sch}
\end{figure}

If $s > 2$, there is no negative cycle in the graph, and thus the necessary condition for the stability of the diamond schemes is $\Delta t = O(\Delta x^2)$. However, eigenvalue analysis reveals that the sufficient condition is $\Delta t = O(\Delta x^3)$. Strictly speaking, for the nonlinear Schr\"{o}dinger equation, the dominant eigenvalue $\lambda_1$ always has an absolute value greater than $1$. Nevertheless, $|\lambda_1| = 1 + O(\Delta t^2)$ holds if and only if $\Delta t = O(\Delta x^3)$. In this case, one has $\lambda_1^{T / \Delta t} \le \exp(T \Delta t) \simeq 1 + T \Delta t$, and hence the schemes can be regarded as stable for sufficiently small time spans $T$.

Fig.~\ref{fig:stable_dt_vs_dx_Sch_simple} and Fig.~\ref{fig:stable_dt_vs_dx_Sch_RK} illustrate the relationship between $\Delta t$ and $\Delta x$ such that $\lambda_1^{1/\Delta t} \le 1.1$ holds for the simple diamond scheme and the higher-order diamond schemes, respectively. Similar to the case of the Boussinesq equation, $\Delta t = O(\Delta x^3)$ is sufficient for stability, and the constant depends on the domain length $b - a$ and on $r^2$.

Fig.~\ref{fig:experiment_Sch1_sol} shows a numerical solution obtained by the simple diamond scheme on the spatial domain $[-24, 24]$, with the parameters $\Delta x = 0.1, \; \Delta t = 2.5 \times 10^{-6}$ and an initial condition of a 2-soliton:
\[
\phi(x, 0) = 3 \sech(3(x + 10)) \exp(3x/4) + \sqrt{6} \sech(\sqrt{6}(x - 10)) \exp(-3x/4).
\]
This setting satisfies the numerically derived stability condition, and the multi-symplectic conservation of the total energy is preserved (see Fig.~\ref{fig:experiment_Sch1_error}).

Fig.~\ref{fig:experiment_Sch1_sol_collapse} presents the results of a numerical experiment with the simple diamond scheme, performed under similar initial conditions but with a large $\Delta t$ that violates the stability condition shown in Fig.~\ref{fig:stable_dt_vs_dx_Sch_simple}. As predicted by the eigenvalue-based stability analysis, the numerical solution diverges.

\begin{figure}[htbp]
    \begin{minipage}[t]{.48\linewidth}
        \begin{center}
            \includegraphics[width=75mm]{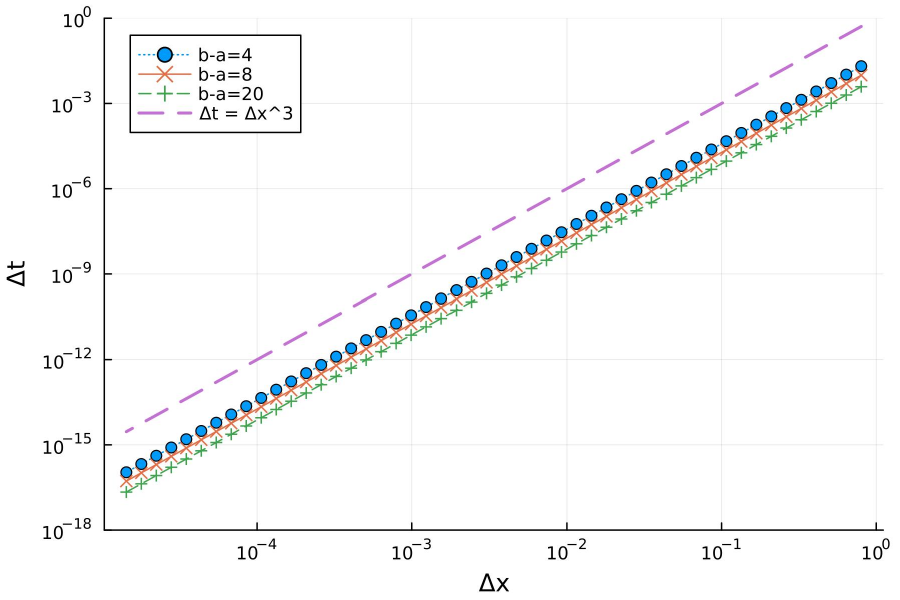}
            \caption{$\Delta t$ vs $\Delta x$ for the stability of the simple diamond scheme of the Schr\"{o}dinger equation, $b-a$ denotes the length of the range $[a, b]$.}
            \label{fig:stable_dt_vs_dx_Sch_simple}
        \end{center}
    \end{minipage}
    \hfill
    \begin{minipage}[t]{.48\linewidth}
        \begin{center}
            \includegraphics[width=75mm]{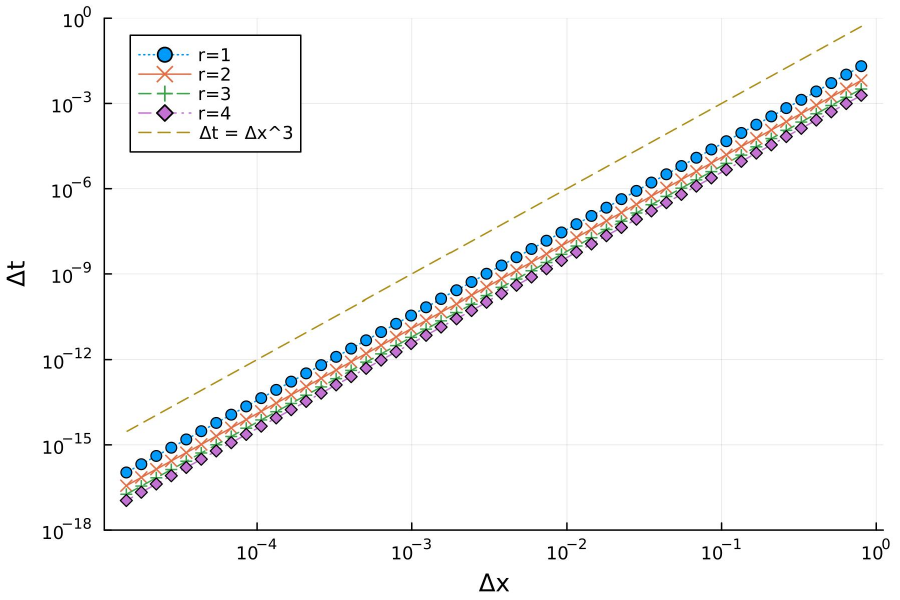}
            \caption{$\Delta t$ vs $\Delta x$ for the stability of the high-order diamond schemes of the Schr\"{o}dinger equation, $r$ denotes the stage of the Runge--Kutta method and $b-a$ is fixed to $4$.}
            \label{fig:stable_dt_vs_dx_Sch_RK}
        \end{center}
    \end{minipage}
\end{figure}

\begin{figure}[htbp]
    \begin{minipage}[t]{.48\linewidth}
        \begin{center}
            \includegraphics[width=75mm]{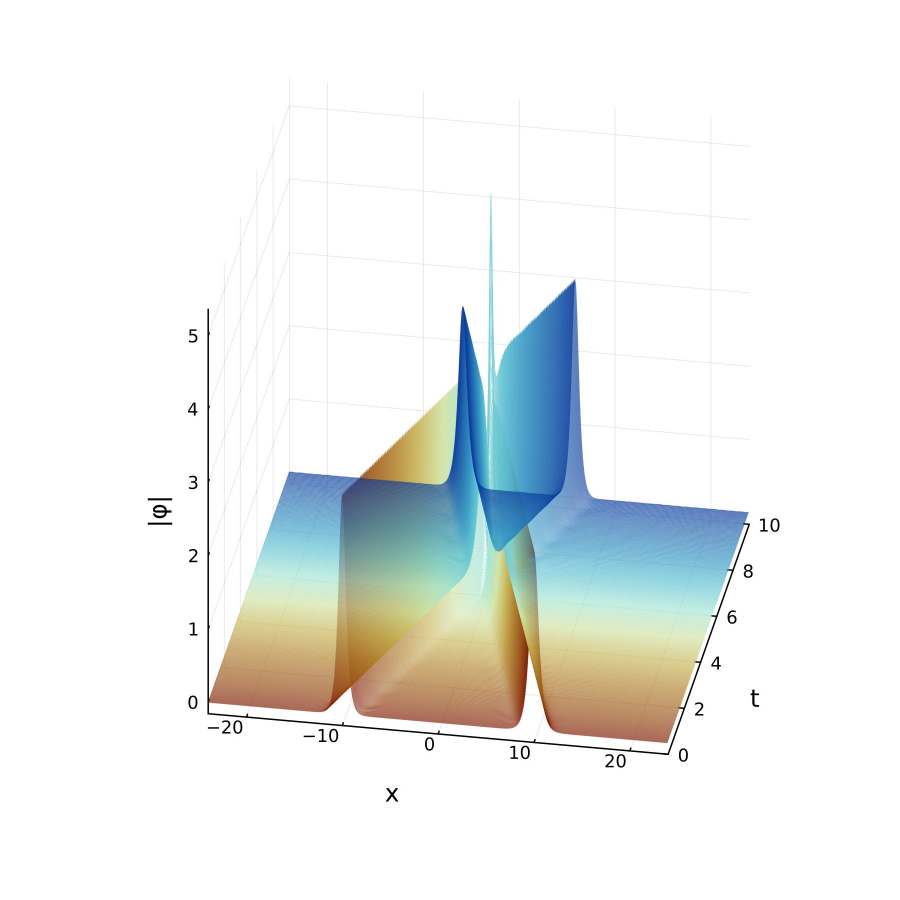}
            \caption{Numerical solutions for the nonlinear Schr\"{o}dinger equation obtained by the simple diamond scheme ($\Delta x = 0.1, \Delta t = 2.5 \times 10^{-6}$)}
            \label{fig:experiment_Sch1_sol}
        \end{center}
    \end{minipage}
    \hfill
    \begin{minipage}[t]{.48\linewidth}
        \begin{center}
            \includegraphics[width=75mm]{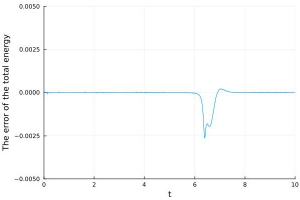}
            \caption{Time evolution of error of the total energy for the nonlinear Schr\"{o}dinger equation ($\Delta x = 0.1, \Delta t = 2.5 \times 10^{-6}$)}
            \label{fig:experiment_Sch1_error}
        \end{center}
    \end{minipage}
\end{figure}

\begin{figure}[htbp]
        \begin{center}
            \includegraphics[width=70mm]{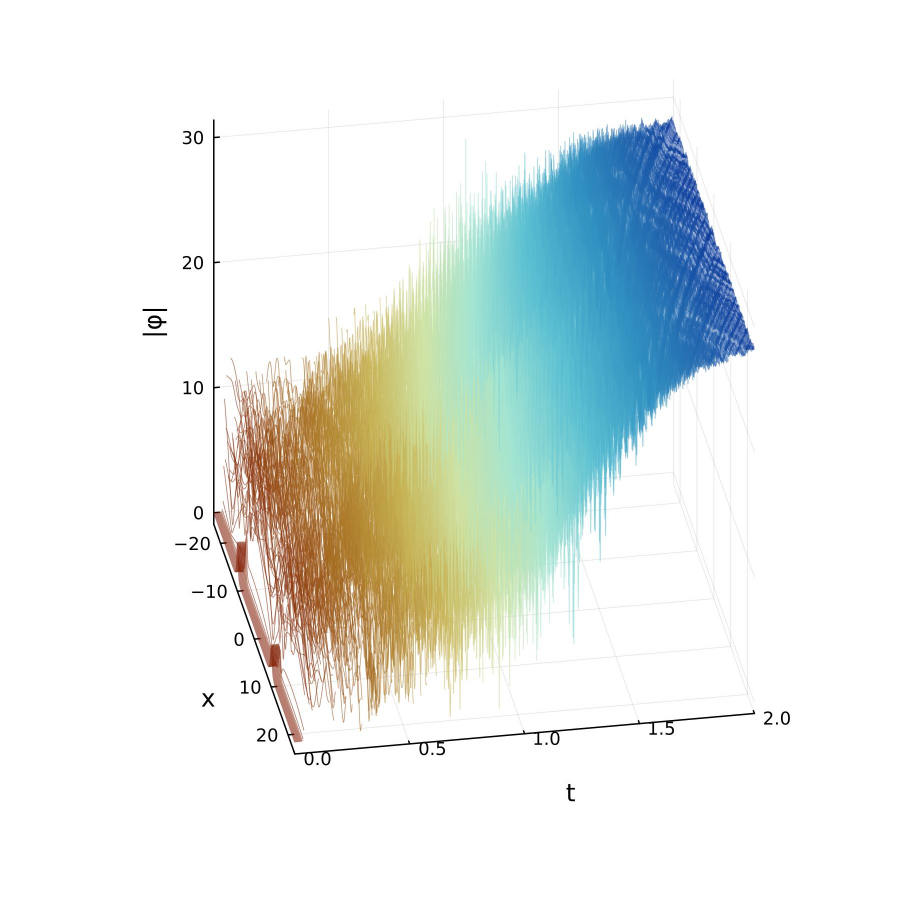}
            \caption{Unstable numerical solutions for the nonlinear Schr\"{o}dinger equation obtained by the simple diamond scheme ($\Delta x = 0.1, \Delta t = 3.33 \times 10^{-6}$)}
            \label{fig:experiment_Sch1_sol_collapse}
        \end{center}
\end{figure}

\section{Conclusion} \label{sec:conclusion}
In this paper, we have proposed a new method to analyze the stability of diamond schemes for multi-symplectic PDEs. Because diamond schemes possess structural features distinct from standard numerical schemes, it is necessary to develop a tailored stability analysis. The proposed method consists of three steps: (1) verifying consistency, (2) checking whether the necessary condition derived from graph stability analysis holds, and (3) deriving the sufficient condition. By applying this method, we classified standard multi-symplectic PDEs according to the stability of diamond schemes. As a result, we found that diamond schemes are stable for the nonlinear Dirac equation under the reasonable condition $\Delta t = O(\Delta x)$, and for the Boussinesq and nonlinear Schr\"{o}dinger equations under the more restrictive condition $\Delta t = O(\Delta x^3)$. On the other hand, our analysis also revealed that diamond schemes are inconsistent and may not be locally and uniquely solvable for certain PDEs, such as the KdV equation, and unconditionally unstable for others, such as the mixed-derivative type Klein--Gordon equation. A summary of the stability analysis results is provided in Table~\ref{tb:stabilities_for_MS_PDEs}.

A possible area of future work is to investigate the practicality of diamond-mesh numerical methods in broader contexts, not only for multi-symplectic PDEs. The advantage of diamond-meshes in computational complexity does not depend on the multi-symplecticness---it suffices that the target PDEs be in first-order PDE systems. The proposed stability analysis method is also useful in such challenges.

\appendix

\section{Further examples of investigating structural consistency for various PDEs} \label{app:examples}
There are other multi-symplectic PDEs in which the simple diamond scheme becomes structurally inconsistent, including:
\begin{itemize}
    \item The Korteweg--de Vries (KdV) equation $u_t + 6uu_x + u_{xxx} = 0$ in the multi-symplectic form given in \cite{Leimkuhler},
    \item The Camassa--Holm equation $u_t - u_{xxt} + 3uu_x = 2u_x u_{xx} + uu_{xxx}$ in the multi-symplectic form given in \cite{YZW16},
    \item The BBM equation $u_t + uu_x + \sigma u_{xxt} = 0$ in the multi-symplectic form given in \cite{LS13}.
    \item Another multi-symplectic formulation of the Hunter--Saxton equation $u_{xxt} + 2u_xu_{xx} + uu_{xxx} = 0$ given in \cite{MCFM17}.
\end{itemize}

In this appendix, we present several examples of multi-symplectic PDEs for which the simple diamond scheme becomes structurally inconsistent.

\subsection{KdV equation}
The Korteweg--de Vries (KdV) equation $u_t + u_{xxx} + uu_x = 0$ can be written in the form of Eq.~\eqref{eq:multi-symplectic} (cf.~\cite{Leimkuhler}) by setting $z = (\psi, u, w, p)$.
\setcounter{equation}{0}
\renewcommand{\theequation}{K\arabic{equation}}
\begin{align}
    \boxed{u_t} + p_x &= 0, \label{eq:KdV1} \\
    -\boxed{\psi_t} - 2w_x&= -\boxed{p} + \boxed{u^2}, \label{eq:KdV2} \\
    2u_x &= 2\boxed{w} , \label{eq:KdV3} \\
    -\psi_x&= -\boxed{u} \label{eq:KdV4}.
\end{align}
\noeqref{eq:KdV1,eq:KdV2,eq:KdV3,eq:KdV4}

The corresponding bipartite graph and its DM decomposition are shown in Figure~\ref{fig:KdV_DM}.
\begin{figure}[htbp]
\centering
\includegraphics{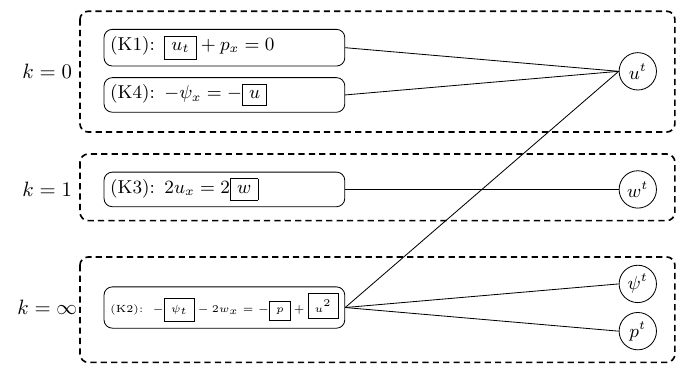}
\caption{Dulmage--Mendelsohn decomposition of the bipartite graph for the KdV equation: equations vs. unknowns \((u^{t},\psi^{t},w^{t},p^{t})\).}\label{fig:KdV_DM}
\end{figure}

In this case, $u^\mathrm{t}$ is overdetermined by Eq.~\eqref{eq:KdV1} and Eq.~\eqref{eq:KdV4}, $\psi^\mathrm{t}$ and $p^\mathrm{t}$ are underdetermined since there is no equation to determine them, and we regard the diamond scheme as structurally inconsistent.

\subsection{Camassa--Holm equation}
The Camassa--Holm equation $u_t - u_{txx} + 3u u_x = 2u_x u_{xx} + uu_{xxx}$ can be written in the multi-symplectic form~\cite{YZW16} by setting $z = (u \; \phi \; w \; \psi \; v)$.

\setcounter{equation}{0}
\renewcommand{\theequation}{CH\arabic{equation}}
\begin{align}
    \frac{1}{2} \boxed{\phi_t} - \frac{1}{2} \boxed{v_t} - \psi_x &= \frac{1}{2} \boxed{w} - \frac{1}{2} \boxed{v},
    \label{eq:CH1} \\
    -\frac{1}{2} \boxed{u_t} - \frac{1}{2} w_x&= 0, \label{eq:CH2}\\
    \frac{1}{2}\psi_x &= \frac{1}{2}\boxed{u} , \label{eq:CH3} \\
    u_x&= \boxed{v} , \label{eq:CH4} \\
    \boxed{u_t}&= -\boxed{uv} + \boxed{\psi}. \label{eq:CH5}
\end{align}
\noeqref{eq:CH1,eq:CH2,eq:CH3,eq:CH4}

The corresponding bipartite graph and its DM decomposition are shown in Figure~\ref{fig:CH_DM}.
\begin{figure}[htbp]
\centering
\includegraphics{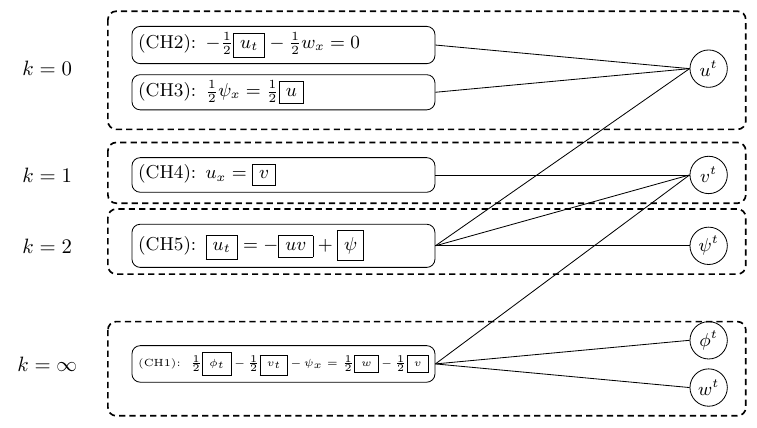}
\caption{Dulmage--Mendelsohn decomposition of the bipartite graph for the Camassa--Holm equation: equations vs. unknowns \((u^{t},\phi^{t},w^{t},\psi^{t},v^{t})\).}\label{fig:CH_DM}
\end{figure}

In this case, first, $u^\mathrm{t}$ is overdetermined by Eq.~\eqref{eq:CH2} and Eq.~\eqref{eq:CH3}. Second, $v^\mathrm{t}$ is well-determined by Eq.~\eqref{eq:CH4}. Third, $\psi^\mathrm{t}$ is well-determined by Eq.~\eqref{eq:CH5}. Finally, $\phi^\mathrm{t}$ and $w^\mathrm{t}$ are underdetermined since only Eq.~\eqref{eq:CH1} involves them.

\subsection{BBM equation}
The BBM equation $u_t + u u_x + \sigma u_{xxt} = 0$ can be written in the multi-symplectic form~\cite{LS13} by setting $z = (\phi \; u \; v \; w \; p)$.

\setcounter{equation}{0}
\renewcommand{\theequation}{B\arabic{equation}}
\begin{align}
    -\frac{1}{2} \boxed{u_t} - p_x &= 0,
    \label{eq:BBM1} \\
    \frac{1}{2} \boxed{\phi_t} - \frac{1}{2} \sigma \boxed{v_t} - \frac{1}{2} \sigma w_x&= \boxed{p} - \frac{1}{2}\boxed{u}^2, \label{eq:BBM2}\\
    \frac{1}{2}\sigma \boxed{u_t} &= \frac{1}{2}\sigma \boxed{w} \label{eq:BBM3}, \\
    \frac{1}{2}\sigma u_x&= \frac{1}{2}\sigma \boxed{v} , \label{eq:BBM4} \\
    \phi_x &= \boxed{u}. \label{eq:BBM5}
\end{align}
\noeqref{eq:BBM1,eq:BBM2,eq:BBM3,eq:BBM4,eq:BBM5}

The corresponding bipartite graph and its DM decomposition are shown in Figure~\ref{fig:BBM_DM}.
\begin{figure}[htbp]
\centering
\includegraphics{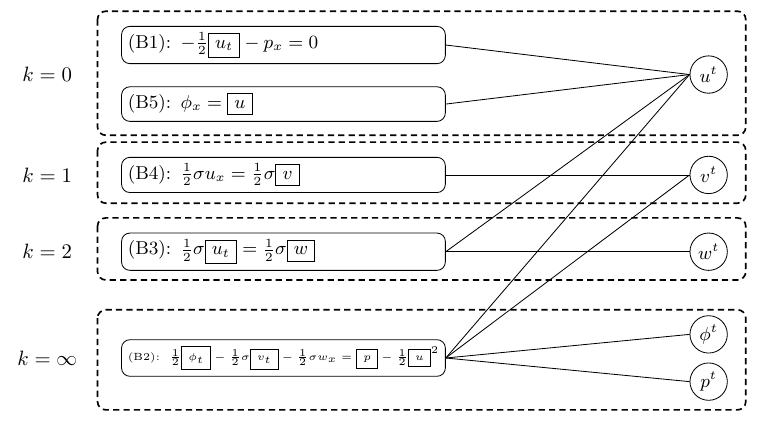}
\caption{Dulmage--Mendelsohn decomposition of the bipartite graph for the BBM equation: equations vs. unknowns \((\phi^{t},u^{t},v^{t},w^{t},p^{t})\).}\label{fig:BBM_DM}
\end{figure}

In this case, $u^\mathrm{t}$ is overdetermined by Eq.~\eqref{eq:BBM1} and Eq.~\eqref{eq:BBM5}.

\subsection{Hunter--Saxton equation}
The Hunter--Saxton equation $u_{xxt} + 2u_x u_{xx} + u u_{xxx} = 0$ can be written in two different multi-symplectic forms~\cite{MCFM17}, which result in either structurally inconsistent diamond schemes. The first multi-symplectic form have been already discussed.

The second multi-symplectic form of the Hunter--Saxton equation is obtained by setting $z = (u \; \beta \; w \; \alpha \; \phi \; \gamma \; P \; r)$.

\setcounter{equation}{0}
\renewcommand{\theequation}{HS'\arabic{equation}}
\begin{align}
    -\frac{1}{2} \boxed{\beta_t} &= -\boxed{\gamma} - \boxed{u\alpha}, \label{eq:HS2_1} \\
    \beta_x &= \boxed{\alpha}, \label{eq:HS2_2}\\
    \frac{1}{2}\boxed{\alpha_t} + \gamma_x &= 0 \label{eq:HS2_3}, \\
    -\beta_x - 2r_x &= 0, \label{eq:HS2_4} \\
    \frac{1}{2} \boxed{u_t} + w_x + P_x &= 0, \label{eq:HS2_5} \\
    -\frac{1}{2} \boxed{\phi_t} &= - \boxed{w} + \frac{1}{2}\boxed{u}^2, \label{eq:HS2_6}\\
    -\phi_x &= -\boxed{u}, \label{eq:HS2_7} \\
    2P_x &= 2\boxed{r}. \label{eq:HS2_8}
\end{align}
\noeqref{eq:HS2_1,eq:HS2_2,eq:HS2_3,eq:HS2_4,eq:HS2_5,eq:HS2_6,eq:HS2_7,eq:HS2_8}

The corresponding bipartite graph and its DM decomposition are shown in Figure~\ref{fig:HS2_DM}.
\begin{figure}[htbp]
\centering
\includegraphics{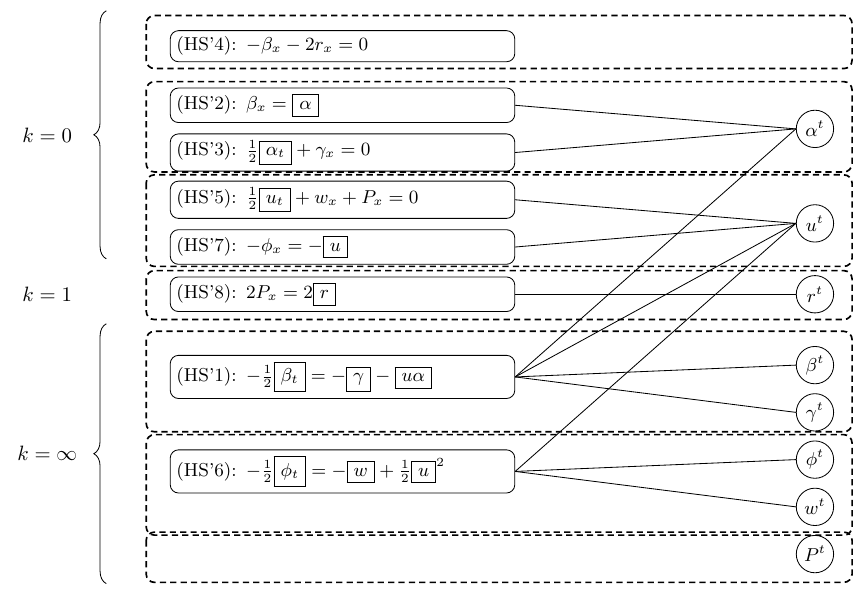}
\caption{Dulmage--Mendelsohn decomposition of the bipartite graph for the second multi-symplectic form of the Hunter--Saxton equation: equations vs. unknowns \((u^{t},\beta^{t},w^{t},\alpha^{t},\phi^{t},\gamma^{t},P^{t},r^{t})\).}\label{fig:HS2_DM}
\end{figure}

In this case, first, Eq.~\eqref{eq:HS2_4} is an isolated equation without unknowns. Second, $\alpha^\mathrm{t}$ is overdetermined by Eq.~\eqref{eq:HS2_2} and Eq.~\eqref{eq:HS2_3}. Third, $u^\mathrm{t}$ is overdetermined by Eq.~\eqref{eq:HS2_5} and Eq.~\eqref{eq:HS2_7}. Fourth, $r^\mathrm{t}$ is well-determined by Eq.~\eqref{eq:HS2_8}. Fifth, $\beta^\mathrm{t}$ and $\gamma^\mathrm{t}$ are underdetermined since only Eq.~\eqref{eq:HS2_1} involves them. Finally, $\phi^\mathrm{t}$ and $w^\mathrm{t}$ are underdetermined since only Eq.~\eqref{eq:HS2_6} involves them. Then, $P^\mathrm{t}$ is an isolated variable without equations.

\section{Inconsistency of the high-order diamond schemes} \label{app:high-order}
We show that the high-order diamond schemes inherit the inconsistency of the simple diamond scheme.
Thus, they are inadequate when the simple diamond scheme has a structural inconsistency.

\begin{thm}
    For the linear multi-symplectic PDEs in which the coefficient matrix of the simple diamond scheme is singular, the high-order diamond schemes using the Runge--Kutta collocation methods also have singular coefficient matrices for updating each diamond.
\end{thm}
\begin{proof}
    We use the same setting and notation as in Section~\ref{subsec:high-order}.
    As shown there, the high-order diamond scheme reduces to the linear system~\eqref{eq:high-order:system}.
    We show that the matrix $Q$ is singular under the assumption.

    If the simple diamond scheme becomes structurally inconsistent, for any complex value $c$, $K - CP$ is singular, since the corresponding bipartite graph $G_\mathrm{bi}$ is constructed considering non-zero patterns of both $K$ and $P$.
    Therefore, for any $c$ there exists a non-zero vector $v_c$ such that $(K - cP)v_c = 0$.

    Now, let $\lambda \neq 0$ be an eigenvalue of $F$, and let $x$ be the corresponding eigenvector. Set $c = 2\lambda / \Delta t$. As above, we can choose $v_c$ such that $(K - cP) v_c = 0$. Let $z = x \otimes x \otimes v_c \in \mathbb{R}^{r^2 d}$. For this $z$,
    \begin{equation*}
        \begin{split}
            Qz &= (I_{r^2} \otimes P)(x \otimes x \otimes v_c) - (I_r \otimes F \otimes \widetilde{K}) (x \otimes x \otimes v_c) - (F \otimes I_r \otimes \widetilde{L}) (x \otimes x \otimes v_c) \\
            &= x \otimes x \otimes (Pv_c) - x \otimes (Fx) \otimes (\widetilde{K} v_c) - (Fx) \otimes x \otimes (\widetilde{L} v_c) \\
            &= (x \otimes x) \otimes (Pv_c - \lambda (\widetilde{K} + \widetilde{L}) v_c) \\
            &= (x \otimes x) \otimes \left(P - \frac{2\lambda}{\Delta t}K\right) v_c = 0.
        \end{split}
    \end{equation*}
    Therefore, $Q$ has a non-trivial kernel and is also singular.
\end{proof}

\section{On enforcing consistency conditions in the overdetermined case} \label{app:enforce}
This subsection discusses the possibility of enforcing additional consistency conditions to resolve the overdetermined case for the example of the advection equation.

Consider the advection equation in the multi-symplectic form~\eqref{eq:AE1}--\eqref{eq:AE3}, for which the simple diamond scheme becomes overdetermined. The update equations on each diamond are given by
\begin{align*}
    &\frac{u^\mathrm{t} - u^\mathrm{b}}{\Delta t} + \frac{w_\mathrm{r} - w_\mathrm{l}}{\Delta x} = 0, \\
    &-\frac{\phi_\mathrm{r} - \phi_\mathrm{l}}{\Delta x} = -\frac{u^\mathrm{t} + u^\mathrm{b} + u_\mathrm{r} + u_\mathrm{l}}{4}, \\
    &-\frac{\phi^\mathrm{t} - \phi^\mathrm{b}}{\Delta t} = 2\frac{u^\mathrm{t} + u^\mathrm{b} + u_\mathrm{r} + u_\mathrm{l}}{4} - \frac{w^\mathrm{t} + w^\mathrm{b} + w_\mathrm{r} + w_\mathrm{l}}{4}.
\end{align*}

To resolve the overdetermination, one may consider enforcing the following consistency condition on the given data at the bottom, left, and right vertices of each diamond:
\begin{equation}
\label{eq:consistency_condition_CE}
    -\frac{\Delta t}{\Delta x} (w_\mathrm{r} - w_\mathrm{l}) + u^\mathrm{b} = \frac{4}{\Delta x} (\phi_\mathrm{r} - \phi_\mathrm{l}) - (u^\mathrm{b} + u_\mathrm{r} + u_\mathrm{l}).
\end{equation}

Demanding the condition~\eqref{eq:consistency_condition_CE} makes all the diamonds in the same time level coupled, since those diamonds share left and right vertices. This makes all the diamonds at the time level coupled and the scheme fully implicit.

In what follows, we consider how the scheme becomes when we enforce the condition~\eqref{eq:consistency_condition_CE}.

Let $U^n, \Phi^n, W^n \in \mathbb{R}^N$ be the vectors of the numerical solutions at time level $n$. The consistency condition~\eqref{eq:consistency_condition_CE} for all diamonds at time step $1$ can be written in the following matrix form:
\[
-\Delta t D^+ W^1 + U^1 = 4 D^+ \Phi^1 + \frac{4}{\Delta x} \begin{pmatrix}
0 \\
0 \\
\vdots \\
\bar{u}
\end{pmatrix} - U^{1/2} - 2M^+ U^1
\]
where $D^+$ is the forward difference matrix, $M^+$ is the forward averaging matrix, and $\bar{u} = \int_0^L u(x, 0) \dd{x}$. Compining this with the second update equations on all diamonds at time step $1$, we obtain the following system of equations for $W^1, \Phi^1$ when $U^1$ is given by the first and third update equations:
\[
\begin{pmatrix}
\Delta t D^+ & 4 D^+ \\
\Delta t I & -4I
\end{pmatrix}
\begin{pmatrix}
W^1 \\
\Phi^1
\end{pmatrix}
=
\begin{pmatrix}
2U^{1/2} + 2M^+ U^1 - \frac{4}{\Delta x} \bm{s} \\
*
\end{pmatrix}.
\]
where $\bm{s} = (0, 0, \ldots, \bar{u})^\top$ and $*$ is some vector determined by $U^1$ and known data. Let $A$ be the coefficient matrix on the left-hand side. Because of the singularity of $D^+$, $A$ is also singular. The left-null vectors of $A$ is multiples of
\[
\bm{v} = [\bm{1}^\top, \bm{0}^\top].
\]
This means that the right-hand side must satisfy the following compatibility condition for the system to be solvable:
\[
\bm{v}^\top
\begin{pmatrix}
2U^{1/2} + 2M^+ U^1 - \frac{4}{\Delta x} \bm{s} \\
*
\end{pmatrix}
=2U^{1/2} + 2M^+ U^1 - \frac{4}{\Delta x} \bm{s} = 0.
\]

Considering the updating formula for $U^1$
\[
U^1 = U^0 - \Delta t D^+ W^{1/2},
\]
we obtain the following compatibility condition for the initial data:
\[
2\bm{1}^\top U^{1/2} + 2\bm{1}^\top (U^0 - \Delta t D^+ W^{1/2}) = \frac{4}{\Delta x} \bar{u}.
\]
Since $\bm{1}^\top D^+ = 0$, this condition reduces to
\[
\bm{1}^\top U^{1/2} + \bm{1}^\top U^0 = \frac{2}{\Delta x} \bar{u}.
\]

Now, since $\ker A$ is spanned by
\[
\begin{pmatrix}
    4\bm{1} \\
    \Delta t \bm{1}
\end{pmatrix},
\]
the solution $(W^1, \Phi^1)$ can be written with the Moore--Penrose pseudoinverse $A^+$ of $A$ as
\[
\begin{pmatrix}
W^1 \\
\Phi^1
\end{pmatrix}
=
A^+
\begin{pmatrix}
2U^{1/2} + 2M^+ U^1 - \frac{4}{\Delta x} \bm{s} \\
*
\end{pmatrix}
+ \lambda
\begin{pmatrix}
4\bm{1} \\
\Delta t \bm{1}
\end{pmatrix},
\]
where $\lambda \in \mathbb{R}$ is arbitrary. Since the value of $\lambda$ does not affect the update of $U^2$, it can be choosen arbitrarily. As a result, by enforcing the consistency condition~\eqref{eq:consistency_condition_CE}, the scheme becomes fully implicit and requires the initial data to satisfy a compatibility condition over the whole spatial domain.

\end{document}